\documentclass[11pt,reqno]{amsart}
\usepackage[a4paper,top=4cm,bottom=4cm,inner=3cm,outer=3cm]{geometry}
\usepackage[pdftex]{graphicx,graphics} 
\usepackage{amsmath,amsfonts}
\usepackage{amssymb,mathrsfs}
\usepackage{color,enumerate}
\usepackage{hyperref} 
\usepackage{epsfig,array} 
\usepackage{mathpazo,tikz}
\tikzset{elegant/.style={smooth,thick,samples=50,cyan}}
\tikzset{eaxis/.style={->,>=stealth}}

\newtheorem{theorem}{Theorem}[section]

\newtheorem{lemma}[theorem]{Lemma}

\newtheorem{remark}[theorem]{Remark}

\newcommand\supp{\mathrm{supp}}
\newcommand\dl{{\underline{\dim}_{\rm loc}}}
\newcommand\du{{\overline{\dim}_{\rm loc}}}
\newcommand\dm{{{\dim}_{\rm loc}}}
\newcommand\El{\underline{E}}
\newcommand\Eu{\overline{E}}

\author{Zhihui Yuan}
\email{yzhh@hust.edu.cn}
\title{Multifractal spectra of Moran measures without local dimension}
\date{}
\begin{document}
	\maketitle

\begin{abstract}
	A measure without local dimension is a measure such that local dimension does not exist for any point in its support.  In this paper, we construct such a class of Moran measures and study their lower and upper local dimensions. We show that the related "free energy" function ($L^q$-spectrum) does not exist. Nevertheless, we can obtain the full Hausdroff and packing dimension spectra for level sets defined by lower and upper local dimensions. They can be viewed as a generalized multifractal formalism. 
	  	
\end{abstract}

\section{Introduction}

Multifractal analysis is a natural framework to finely describe geometrically the heterogeneity in the distribution at small scales of the measures on a metric space. The multifractal formalism aims at expressing the dimension of the level sets in terms of the Legendre transform of some "free energy" function. The general setting is as follows.

Assume $\mu$ is a probability measure supported on a compact metric space $(X,d)$. The local  dimension of $\mu$ at $x$ is
$$
\dm(\mu,x):=\lim_{r\to 0^+}\frac{\log \mu(B(x,r))}{\log r},
$$
if the limit exists. For $\alpha\in \mathbb{R}$, define the level set as 
$$
E(\mu,\alpha):=\{x\in \supp(\mu):\dm(\mu,x)=\alpha \}.
$$
The multifractal analysis is to determine these $\alpha$ such that $E(\mu,\alpha)\ne \emptyset$ and compute the dimension of $E(\mu,\alpha)$ which is called dimension spectra. In many classical cases, the dimension spectra is related to the so-called ``free energy" function $\tau_{\mu}$ (also called by $L^q$-spectrum), which is defined as follows:
\begin{equation*}\label{Lq}
\tau_{\mu}(s)=\lim_{r\rightarrow 0}\frac{\log \sup\{\sum_{i}(\mu(B_i))^s\}}{\log r},
\end{equation*}
where the supremum is taken over all families of disjoint closed balls $B_i=B(x_i,r)$ of radius $r$ with centers in $\supp(\mu)$.  Of course it is only meaningful when the limit exists. We will say that the multifractal formalism is valid if 
$$
\dim_H E(\mu,\alpha)=\tau_\mu^\ast (\alpha),
$$
where $f^*$ denotes the Legendre transform (i.e. $f^*(\alpha)=\inf_{q\in\mathbb{R}}\{\alpha q-f(q)\}$) and a negative dimension  means that the set is empty.

For many classical systems, such as Cookie-Cutter systems with a Gibbs measure~\cite{Rand1989}, subshift of finite type with a weak Gibbs measure ~\cite{FO2003}, self-similar set with a self-similar measure~\cite{FL2009},random weak Gibbs measure~\cite{Yuan2017} and also ~\cite{BMP1992,CM1992,Falconer1997,Testud2006}, the multifractal formalism is known to hold.  On the other hand, it is also known that for some ``bad" measure, the multifractal formalism does not hold. In \cite{BBH2002}, the authors  construct a measure such that  the multifractal formalism is nowhere valid.  The breakdown of multifractal formalism is attributed to the fact that $\tau_{\mu}$ does not eixst.

When $\tau_\mu$ does not exist, one can define 
\begin{equation*}\label{Lql}
\underline{\tau}_{\mu}(s)=\liminf_{r\rightarrow 0}\frac{\log \sup\{\sum_{i}(\mu(B_i))^s\}}{\log r},\ \ \text{ and}\ \  \overline{\tau}_{\mu}(s)=\limsup_{r\rightarrow 0}\frac{\log \sup\{\sum_{i}(\mu(B_i))^s\}}{\log r}
\end{equation*}
respectively.

 In \cite{BP2013}, the authors provide an example of a measure on the interval $[0,1]$ for which the functions $\underline{\tau}_{\mu}$ and $\overline{\tau}_{\mu}$ differ and the Hausdorff dimensions of the sets $E(\mu,\alpha)$ are given by the Legendre transform of $\overline{\tau}_{\mu}$, and their packing dimensions by the Legendre transform of $\underline{\tau}_{\mu}$ on a subset of the admissible interval. \cite{Shen2015} intensifies \cite{BP2013} such that the function $\underline\tau_{\mu}$ and $\overline{\tau}_{\mu}$ can be real analytic.
 
 In \cite{Qu2018}, the author studies  the density of states measure of certain Sturm Hamiltonians. He shows that such class of measures is always exact upper and lower dimensional. Here, we call a measure $\mu$ is exact upper dimensional if $\mu$-almost every point $x$, $\du(\mu,x)=C_1$ for some constant $C_1$, also we call a measure $\mu$ is exact lower dimensional if $\mu$-almost every point $x$, $\dl(\mu,x)=C_2$ for some constant $C_2$. His example suggests the possibility that the local dimension might do not exist for any point in the support. This motivates our construction of the present paper. Indeed,
 our model can serve as a toy model of it, but our results may throw light on its multifractal analysis.
 
 When the local dimension does not exist, we can define the lower and upper local dimensions of $\mu$, i.e.
$$\dl(\mu,x)=\liminf_{r\to 0^+}\frac{\log \mu(B(x,r))}{\log r}\ \  \text{ and }\ \  \du(\mu,x)=\limsup_{r\to 0^+}\frac{\log \mu(B(x,r))}{\log r}.$$ 
For $\alpha\in \mathbb{R}$, we define  
$$\El(\mu,\alpha)=\{x\in \supp(\mu) :\dl(\mu,x)=\alpha\},$$
$$\Eu(\mu,\alpha)=\{x\in \supp(\mu) :\du(\mu,x)=\alpha\}.$$
Then we have 
$$E(\mu,\alpha)=\El(\mu,\alpha)\cap\Eu(\mu,\alpha).$$
We can further decompose $\El(\mu,\alpha)$ and $\Eu(\mu,\alpha)$ in the following way.
 For any $\alpha'\ge \alpha,$ define
 $$E(\mu,\alpha,\alpha'):=\{x\in \supp(\mu):\dl(\mu,x)=\alpha,\du(\mu,x)=\alpha'\},$$
 Then we have $E(\mu,\alpha,\alpha')=\El(\mu,\alpha)\cap\Eu(\mu,\alpha')$ and 
 $$
 \El(\mu,\alpha)=\bigcup_{\alpha'\ge\alpha}E(\mu,\alpha,\alpha'), \ \  \  \ \Eu(\mu,\alpha')=\bigcup_{\alpha\le\alpha'}E(\mu,\alpha,\alpha').
 $$
 
 In this paper, we will construct a class of Moran measures   such that the free energy function does not exist, and that the local dimension  does not exist  for each point in its support. Then, we will compute the dimensions of $\El(\mu,\alpha),\ \Eu(\mu,\alpha)$ and $E(\mu,\alpha,\alpha')$ and indicate that some generalized multifractal formalism holds.

The paper is organized as follows. In section 2, we present our model and state main results. Section 3 provides the basic properties for a more general model which will be useful in the section 4 where we prove our results.

\section{The Construction and main results}\label{section: The Construction and main results}

Let $\mathcal{A}=\{0,1\}$ and $\mathcal{A}^*$ be the set of all finite words on the alphabet $\mathcal{A}$.

Fix two real numbers $A,B$ with $A>B>2$. Let $\mathcal{N}:=\{N_i\}_{i\in\mathbb{N}}$ be an increasing sequence of integers. Define a set $X(A,B,\mathcal{N})$ as follows.

\begin{enumerate}
	\item [Step 1:] Let $I=[0,1]$. For $n=1$, define:
	$I_{0}=[0,1/A]\subset I,$ and $I_{1}=[1-1/A,1]\subset I.$
	\item [Step 2:] For $n\in\mathbb{N}$ with $n>1$, we assume that for all $w\in \mathcal{A}^{n-1}$, the set $I_{w}$ has been defined. Let $x_w$ be the left endpoint of $I_{w}$. 
	\begin{itemize}
		\item If $N_{2i}<n\leq N_{2i+1}$ for some $i\in\mathbb{N}$, define 
		$I_{w\ast 0}=[x_w,x_w+\frac{|I_{w}|}{A}]$ and $I_{w\ast 1}=[x_w+|I_{w}|-\frac{|I_{w}|}{A},x_w+|I_{w}|]$.
        \item If $N_{2i+1}<n\leq N_{2i+2}$ for some $i\in\mathbb{N}$, define 
        $I_{w\ast 0}=[x_w,x_w+\frac{|I_{w}|}{B}]$ and $I_{w\ast 1}=[x_w+|I_{w}|-\frac{|I_{w}|}{B},x_w+|I_{w}|]$.
	\end{itemize}
	\item [Step 3:] Define $X(A,B,\mathcal{N})=\cap_{n\in\mathbb{N}}\cup_{w\in\mathcal{A}^{n}}I_{w}$.
\end{enumerate}

Given two real numbers $p,q$ with $0<p,q\leq 1/2$, we will distribute a probability measure $\mu_{(p,q,\mathcal{N})}$ on $X(A,B,\mathcal{N})$ as follows:

\begin{enumerate}
	\item [Step 1:] Let $\mu_{(p,q,\mathcal{N})}(I)=1$. For $n=1$, define:
	$\mu_{(p,q,\mathcal{N})}(I_0)=p$ and $\mu_{(p,q,\mathcal{N})}(I_1)=1-p$
	\item [Step 2:] For $n\in\mathbb{N}$ with $n>1$, we assume that for all $w\in \mathcal{A}^{n-1}$, the set $\mu_{(p,q,\mathcal{N})}(I_{w})$ has been defined.
	\begin{itemize}
		\item If $N_{2i}<n\leq N_{2i+1}$ for some $i\in\mathbb{N}$, define 
          $\mu_{(p,q,\mathcal{N})}(I_{w\ast0})=p\mu_{(p,q,\mathcal{N})}(I_{w})$ and $\mu_{(p,q,\mathcal{N})}(I_{w\ast 1})=(1-p)\mu_{(p,q,\mathcal{N})}(I_{w})$.
		\item If $N_{2i+1}<n\leq N_{2i+2}$ for some $i\in\mathbb{N}$, define 
          $\mu_{(p,q,\mathcal{N})}(I_{w\ast0})=q\mu_{(p,q,\mathcal{N})}(I_{w})$ and $\mu_{(p,q,\mathcal{N})}(I_{w\ast 1})=(1-q)\mu_{(p,q,\mathcal{N})}(I_{w})$.
	\end{itemize}
	\item [Step 3:] Such a set function $\mu_{(p,q,\mathcal{N})}$ defined on $\cup_{n\in\mathbb{N}}\mathcal{A}^{n}$ can be extended to a probability measure on the whole $\sigma$-algebra by measure extension theorem. We still denote the measure by $\mu_{(p,q,\mathcal{N})}$.
\end{enumerate}

In this paper, we always make the following assumption:
\begin{equation}\label{ass}
\lim_{i\to\infty}\frac{N_{i+1}}{N_{i}}=\infty\ \ \text{ and }\ \ -\frac{\log p}{\log A}<-\frac{\log (1-q)}{\log B}.
\end{equation}
Furthermore, if there is no conflict we always denote $X:=X(A,B,\mathcal{N})$ and $\mu:=\mu_{(p,q,\mathcal{N})}$.

Given $0<p,\tilde p<1,$ define the mixed entropy function
$$
H(\tilde p, p):=-\tilde p \log p- (1-\tilde p)\log(1-p).
$$
We also write $H(p)=H(p,p)$.

\begin{theorem}\label{dim-loc-bound}
	\begin{enumerate}
		\item For each point $x\in X={\rm supp}(\mu)$, one has $$-\frac{\log (1-p)}{\log A}\leq\underline{\dim}_{\rm loc}(\mu,x)\leq -\frac{\log p}{\log A}$$ and $$-\frac{\log (1-q)}{\log B}\leq \overline{\dim}_{\rm loc}(\mu,x)\leq-\frac{\log q}{\log B}.$$
		\item  $\dl(\mu,x)<\du(\mu,x)$ for all $x\in X={\rm supp}(\mu)$.  
				\item For $\mu$-almost every $x\in X={\rm supp}(\mu)$, one has 
		$$\underline{\dim}_{\rm loc}(\mu,x)=\frac{H(p )}{\log A},\ \ \text{ and }\ \ \overline{\dim}_{\rm loc}(\mu,x)=\frac{H(q )}{\log B}.$$ 
		\item $\dim_{H} X=\frac{\log 2}{\log A}<\dim_{P} X=\frac{\log 2}{\log B}$.
	\end{enumerate}
\end{theorem}

\begin{remark}
{\rm
(i)  If $p=1/2$, then for any $x\in X$,  $\underline{\dim}_{\rm loc}(\mu,x)=\log 2/\log A$, hence there is no multifractal analysis for $\El(\mu,\alpha)$, since $\El(\mu,\log 2/\log A)=X.$ By the same reason, if $q=1/2$, there is no multifractal analysis for $\Eu(\mu,\alpha)$.

(ii) The assumption $B>2$ is to ensure that the strong separation condition hold. Using this property, it is equivalent to consider balls and cylinders.

(iii)  The second inequality of ~\eqref{ass} ensures item 2 of Theorem~\ref{dim-loc-bound}, so that this model can provide an example of measure such that the local dimension does not exist for each point in its support. Such nontrivial examples have not been given before to the author's best knowledge.

(iv)  (3) is kind of Young's dimension formula. 
}
\end{remark}

Next we study the dimensions of the level sets.

Define the function
\begin{eqnarray*}
	\beta_1:&\mathbb{R}\to& \mathbb{R}\\
	&s\mapsto&-\frac{\log(p^s+(1-p)^s)}{\log A},
\end{eqnarray*}

then 
$$\beta_{1}'(+\infty)=\frac{H(0,p)}{\log A}\le\beta_{1}'(1)=\frac{H(p,p)}{\log A}\le
	\beta_{1}'(0)=\frac{H(1/2,p)}{\log A}\le\beta_{1}'(-\infty)=\frac{H(1,p)}{\log A}.$$
	and define $\tilde{\beta}_1$ which is a revised function of $\beta_{1}$ (see figure \ref{fig:beta1}),
	$$\tilde{\beta}_1(s)=		\begin{cases}
	\beta_1(s)& s\in (-\infty,0]\cup[1,+\infty),\\
	(1-s)\beta_1(0)+s\beta_1(1)& \, s\in (0,1).
	\end{cases}$$ 
	
\begin{figure}[h]  
	\begin{minipage}[t]{0.5\linewidth}  
		\centering  
		\includegraphics[width=2.8in]{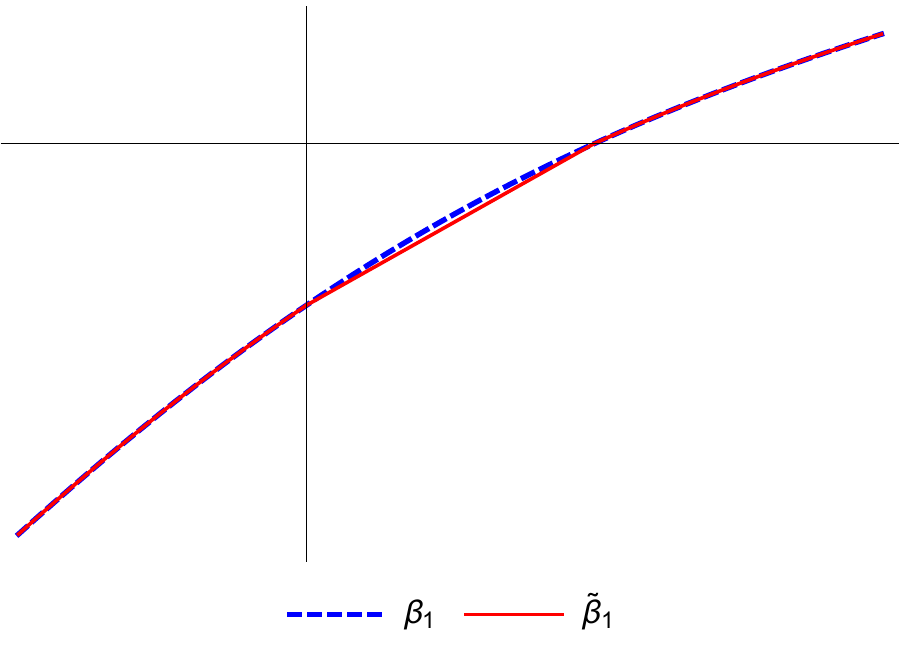}  
		\caption{} 
		\label{fig:beta1}  
	\end{minipage}
	\begin{minipage}[t]{0.5\linewidth}  
		\centering  
		\includegraphics[width=2.8in]{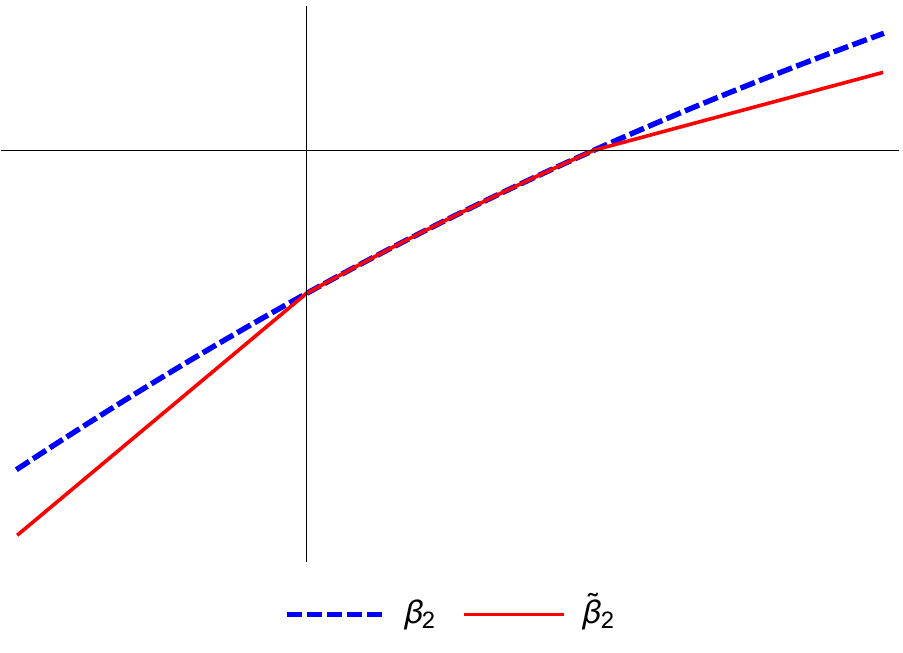}  
		\caption{} 
		\label{fig:beta2}  
	\end{minipage}  
\end{figure} 

Also, define	
\begin{eqnarray*}
	\beta_2:&\mathbb{R}\to& \mathbb{R}\\
	&s\mapsto&-\frac{\log(q^s+(1-q)^s)}{\log B},
\end{eqnarray*}
then 
$$\beta_{2}'(+\infty)=\frac{H(0,q)}{\log B}\le\beta_{2}'(1)=\frac{H(q,q)}{\log B}\le
\beta_{2}'(0)=\frac{H(1/2,q)}{\log B}\le\beta_{2}'(-\infty)=\frac{H(1,q)}{\log B}.$$
and define the revised function of $\beta_{2}$ (see figure \ref{fig:beta2})
$$\tilde{\beta}_2(s)=\begin{cases}
\beta_2(0)+s\beta'_2(0)& \, s\in (-\infty,0),\\
\beta_2(s)& s\in [0,1],\\
\beta_2(1)+\beta'_2(1)(s-1)& \, s\in (1,+\infty).
\end{cases}$$
 
The two functions $\beta_{1},\beta_{2}$ can be seen as the $L^q$-spectrum if we just use the data $(A,p)$ or $(B,q)$.  

\begin{theorem}\label{main}
	\begin{enumerate}
		\item If $\alpha\notin [\beta_{1}'(+\infty),\beta_{1}'(-\infty)]$,\  $\El(\mu,\alpha)=\emptyset$.\\ If $\alpha\notin [\beta_{2}'(+\infty),\beta_{2}'(-\infty)]$, $\Eu(\mu,\alpha)=\emptyset$.
		\item For $\alpha\in [\beta_{1}'(+\infty),\beta_{1}'(-\infty)]$, we have
		$$
		\dim_H \El(\mu,\alpha)=\tilde \beta_1^\ast(\alpha) \ \ \ \text{ and }\ \ \ \dim_{P}\underline{E}(\mu,\alpha)=\frac{\log 2}{\log B}.
		$$
		
		\item For $\alpha\in [\beta_{2}'(+\infty),\beta_{2}'(-\infty)]$, we have
		\begin{equation*}
		\dim_{H}\overline{E}(\mu,\alpha)=\min\{\frac{\log 2}{\log A},\beta_2^*(\alpha)\}
		\end{equation*}
		and 
		\begin{equation*}
		\dim_{P}\overline{E}(\mu,\alpha)=\tilde{\beta}_2^*(\alpha)
		\end{equation*}
	\end{enumerate}
\end{theorem}

Here we draw figure~\ref{fig:local-dimension} and give some remarks to illustrate theorem~\ref{main}.
\begin{figure}[htp!]
	\centering
	\includegraphics[width=0.7\linewidth]{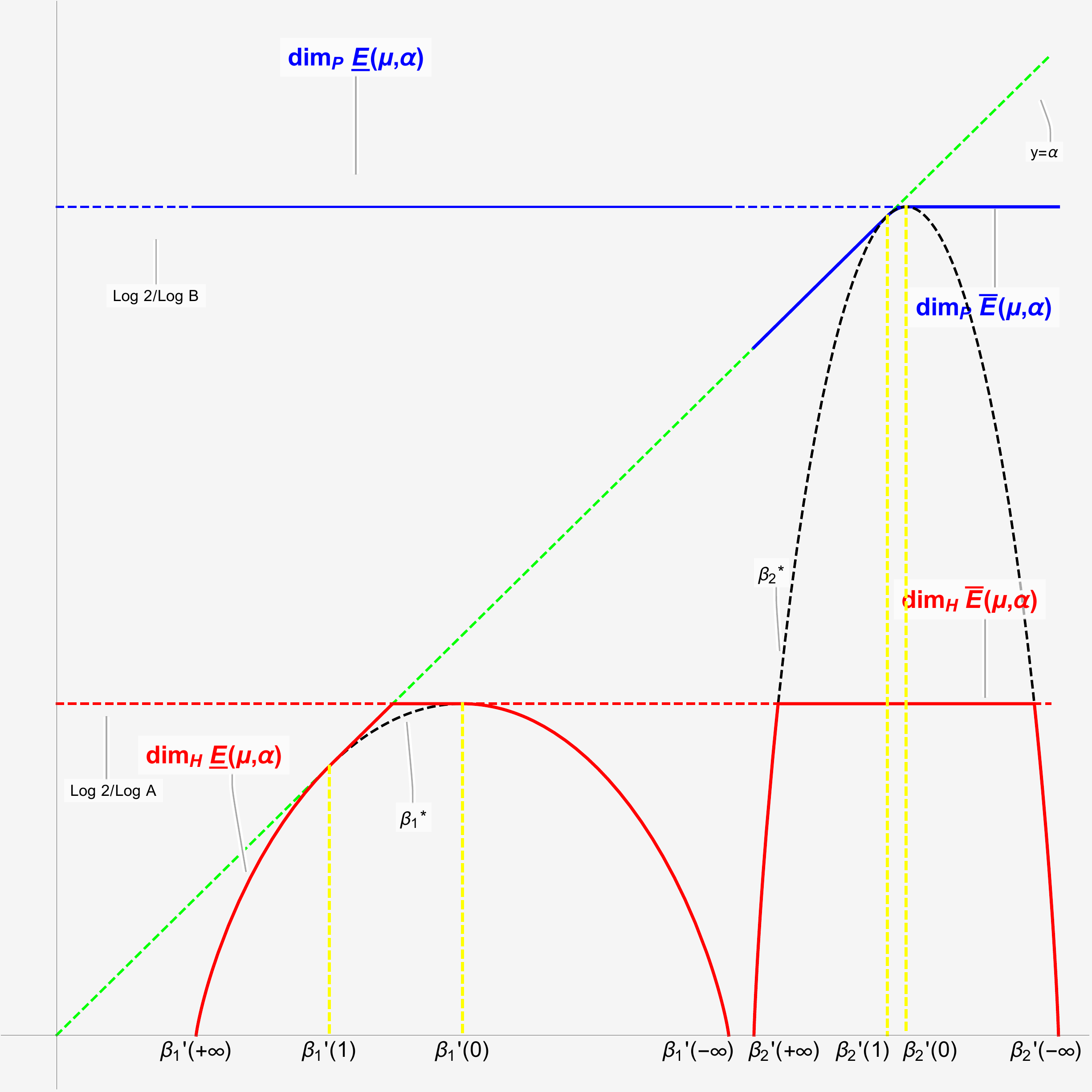}
	\caption{The Hausdorff and packing dimension for $\El(\mu,\alpha)$ and $\Eu(\mu,\alpha)$}
	\label{fig:local-dimension}
\end{figure}

\begin{remark}
{\rm
(i)\  We will compute the explicit formula of the Hausdorff dimension spectrum as 
\begin{equation*}
		\dim_{H}\El(\mu,\alpha)=
		\begin{cases}
		\beta_1^*(\alpha)& \alpha\in[\beta_{1}'(+\infty),\beta_{1}'(1))\cup [\beta_{1}'(0),\beta_{1}'(-\infty)],\\
		\min\{\alpha,\frac{\log 2}{\log A}\}& \,\alpha\in [\beta_{1}'(1),\beta_{1}'(0))
		\end{cases}
		\end{equation*}
		which turns out to be $\tilde{\beta}_1^*(\alpha)$.
		
\rm
(ii)\  We will also compute the explicit formula of the Hausdorff dimension spectrum as 
		\begin{equation*}
         \dim_{P}\overline{E}(\mu,\alpha)=
             \begin{cases}
              \alpha& \alpha\in[\beta_{2}'(+\infty),\beta_{2}'(1)),\\
              \beta_2^*(\alpha)& \,\alpha\in [\beta_{2}'(1),\beta_{2}'(0))\\
              \frac{\log 2}{\log B}& \,\alpha\in [\beta_{2}'(0),\beta_{2}'(-\infty)].
             \end{cases}
         \end{equation*}
         which turns out to be $\tilde{\beta}_2^*(\alpha)$.
}
\end{remark}	 

At last, we compute the dimension of $E(\mu,\alpha,\alpha')$.

	\begin{figure}[h]  
	\begin{minipage}[t]{0.5\linewidth}  
		\centering  
		\includegraphics[width=2.8in]{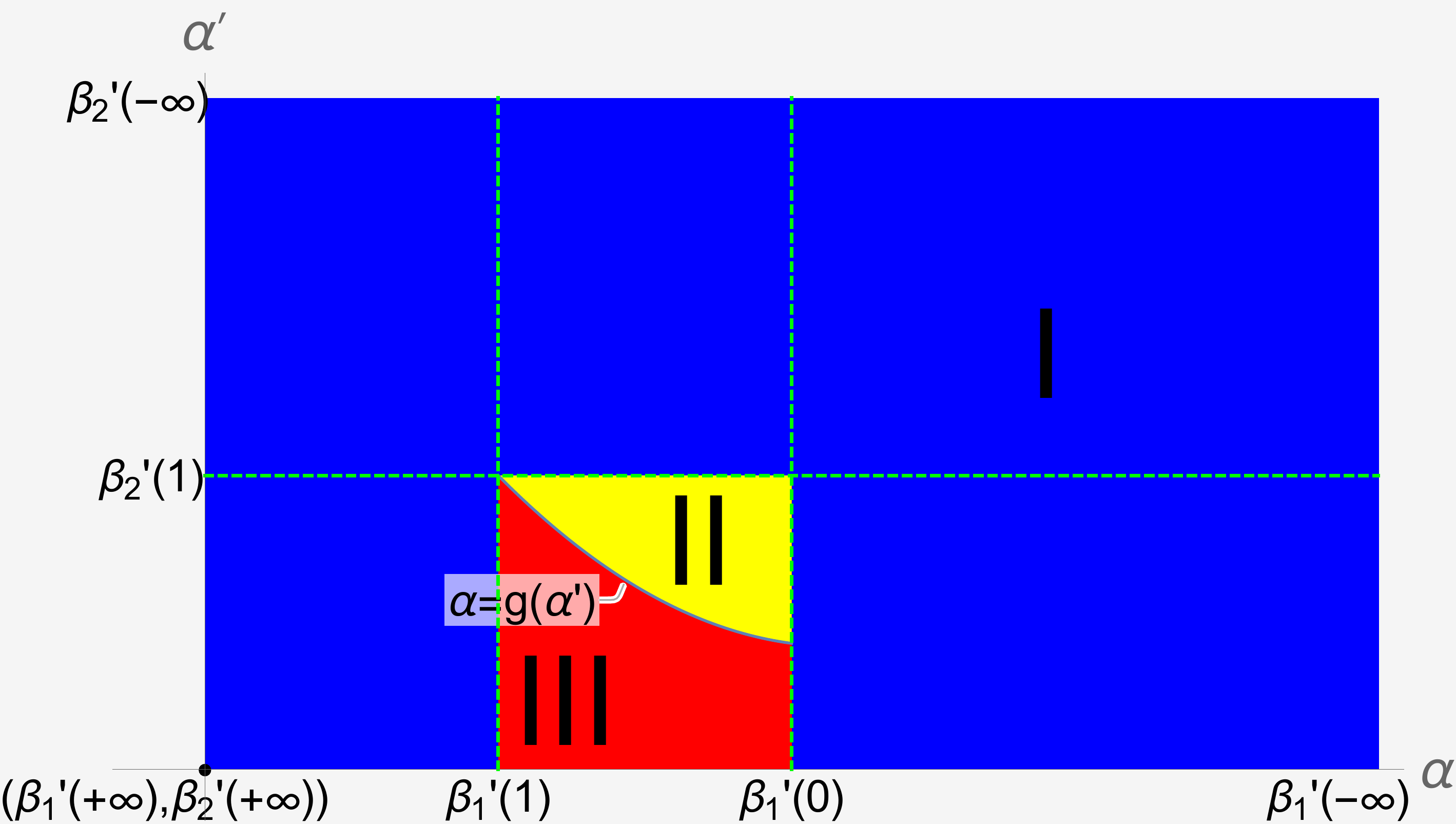}  
		\caption{Domain for Hausdorff dimension}  
		\label{fig:dom1}  
	\end{minipage}
	\begin{minipage}[t]{0.5\linewidth}  
		\centering  
		\includegraphics[width=2.8in]{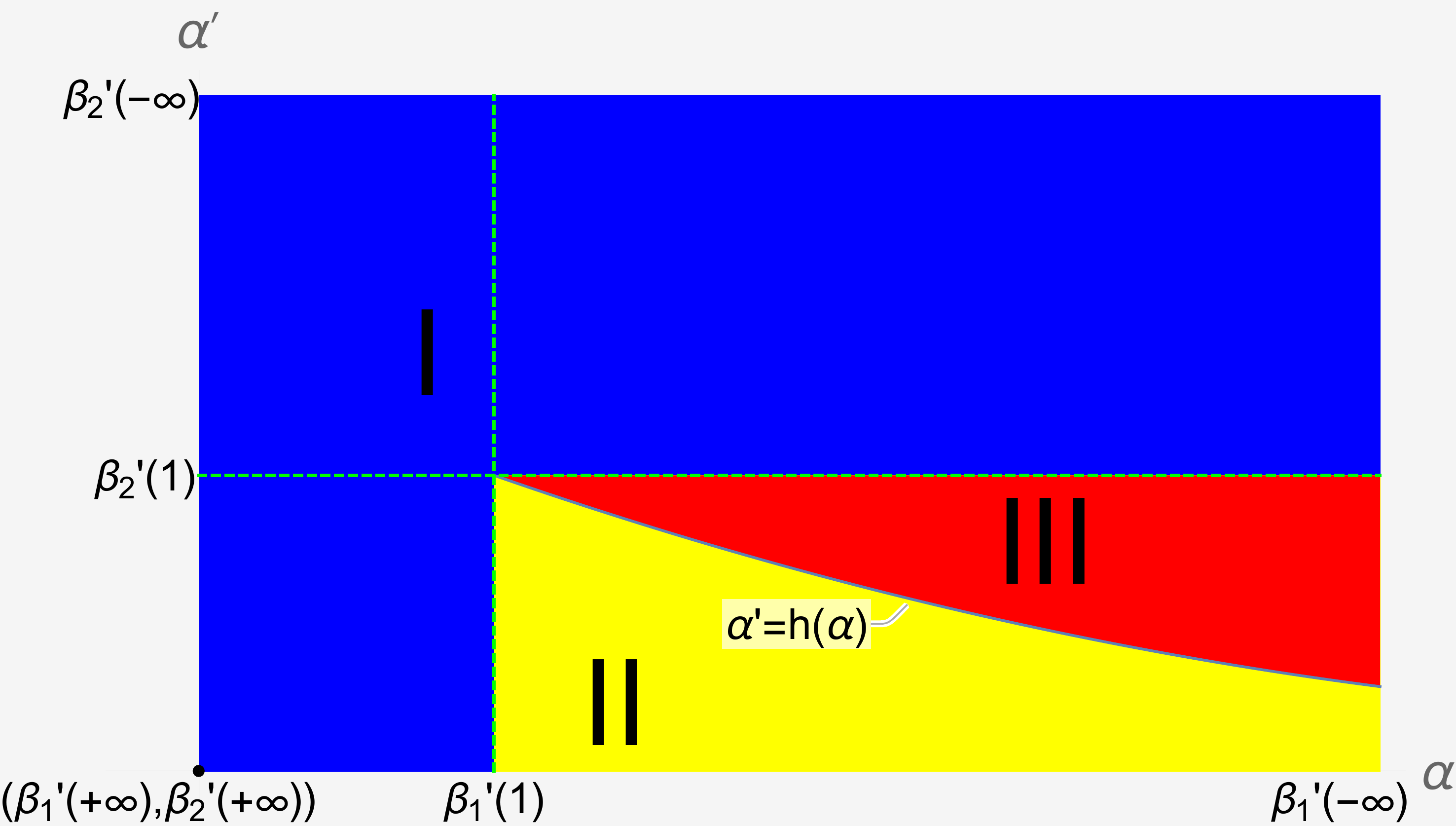}  
		\caption{Domain for Packing dimension}  
		\label{fig:dom2}  
	\end{minipage}  
\end{figure}

\begin{theorem}\label{main2}
Assume $(\alpha,\alpha')\in [\beta_{1}'(+\infty),\beta_{1}'(-\infty)]\times[\beta_{2}'(+\infty),\beta_{2}'(-\infty)]$. 
	\begin{enumerate}
		\item  we have
		\begin{itemize}
			\item if $(\alpha,\alpha')\notin [\beta'_1(1),\beta'_1(0)]\times[\beta'_2(+\infty),\beta'_2(1))$, i.e. $(\alpha,\alpha')$ in domain $I$ (the Blue part) of figure \ref{fig:dom1}\\	$$\dim_{H}(E(\mu,\alpha,\alpha'))=\min\{\dim_{H}\El(\mu,\alpha),\dim_{H}\Eu(\mu,\alpha')\}.$$
			\item if $\alpha'\in[\beta'_2(+\infty),\beta_{2}'(1))$ with $g(\alpha')\leq\beta_1'(0)$, where we take the tangent to the graph $\beta_1^*$ passing through the point $(\alpha',\beta_2^*(\alpha'))$ and denote the point of tangency by $(g(\alpha'),\beta_1^*(g(\alpha')))$. (One can refer figure \ref{fig:Thm-1} for the definition of the function $g$.)
			
			For $\alpha\in[g(\alpha'),\beta_1'(0))$, i.e. $(\alpha,\alpha')$ in domain $II$ (the yellow part) of figure \ref{fig:dom1}\\ 
			$$\dim_{H}(E(\mu,\alpha,\alpha'))=\min\{\frac{\log 2}{\log A},s_1\alpha-\beta_1(s_1)\},$$
			 where $s_1\in[0,1]$ satisfies $\beta_1'(s_1)=g(\alpha')$.
			\item otherwise,  i.e. $(\alpha,\alpha')$ in domain $III$ (the red part) of figure \ref{fig:dom1} 
			$$\dim_{H}(E(\mu,\alpha,\alpha'))=\min\{\beta_1^*(\alpha),\beta_2^*(\alpha')\}.$$
		\end{itemize}
		\item For $(\alpha,\alpha')\in [\beta_{1}'(+\infty),\beta_{1}'(-\infty)]\times[\beta_{2}'(+\infty),\beta_{2}'(-\infty)]$, we have
		\begin{itemize}
			\item if $(\alpha,\alpha')\notin [\beta'_1(1),\beta'_1(-\infty)]\times[\beta'_2(+\infty),\beta'_2(1))$, $(\alpha,\alpha')$ in domain $I$ (the blue part) of figure \ref{fig:dom2} \\
				$$\dim_{P}(E(\mu,\alpha,\alpha'))=\dim_{P} \Eu(\mu,\alpha').$$
			\item if $\alpha\in[\beta'_1(1),\beta'_1(-\infty)]$ and $\alpha'\in[\beta_2'(+\infty),h(\alpha))$. Here, we take the tangent to the graph $\beta_2^*$ passing through the point $(\alpha,\beta_1^*(\alpha))$ and denote the point of tangency by $(h(\alpha),\beta_2^*(h(\alpha)))$. (One can refer figure \ref{fig:Thm-2} for the definition of the function $h$.)
			Then take $s_2\in(1,+\infty]$ such that $\beta_2'(s_2)=h(\alpha)$.
			In this case, $(\alpha,\alpha')$ in domain $II$ (the Yellow part) of figure \ref{fig:dom2}
			$$\dim_{P}(E(\mu,\alpha,\alpha'))=s_2\alpha'-\beta_2(s_2).$$
			\item otherwise, $(\alpha,\alpha')$ in domain $III$ (the red part) of figure \ref{fig:dom2}
			$$\dim_{P}(E(\mu,\alpha,\alpha'))=\beta_2^*(\alpha').$$
		\end{itemize}
	\end{enumerate}
\end{theorem}

We now draw figure~\ref{fig:Thm-1} (the red line) for $\dim_{H}(E(\mu,\alpha,\alpha'))$ when  $\alpha\in [\beta'_1(1),\beta'_1(0)]$ for a fixed $\alpha'\in[\beta'_2(+\infty),\beta_{2}'(1))$ with $g(\alpha')\leq\beta_1'(0)$ and  figure~\ref{fig:Thm-2} (the red line) for $\dim_{P}(E(\mu,\alpha,\alpha'))$ when  $\alpha'\in [\beta'_2(+\infty),\beta'_2(1)]$ with a fixed $\alpha\in(\beta_{2}'(1),\beta_{2}'(-\infty)]$ to illustrate theorem~\ref{main2}. 
	\begin{figure}[h]  
	\begin{minipage}[t]{0.5\linewidth}  
		\centering  
		\includegraphics[width=2.8in]{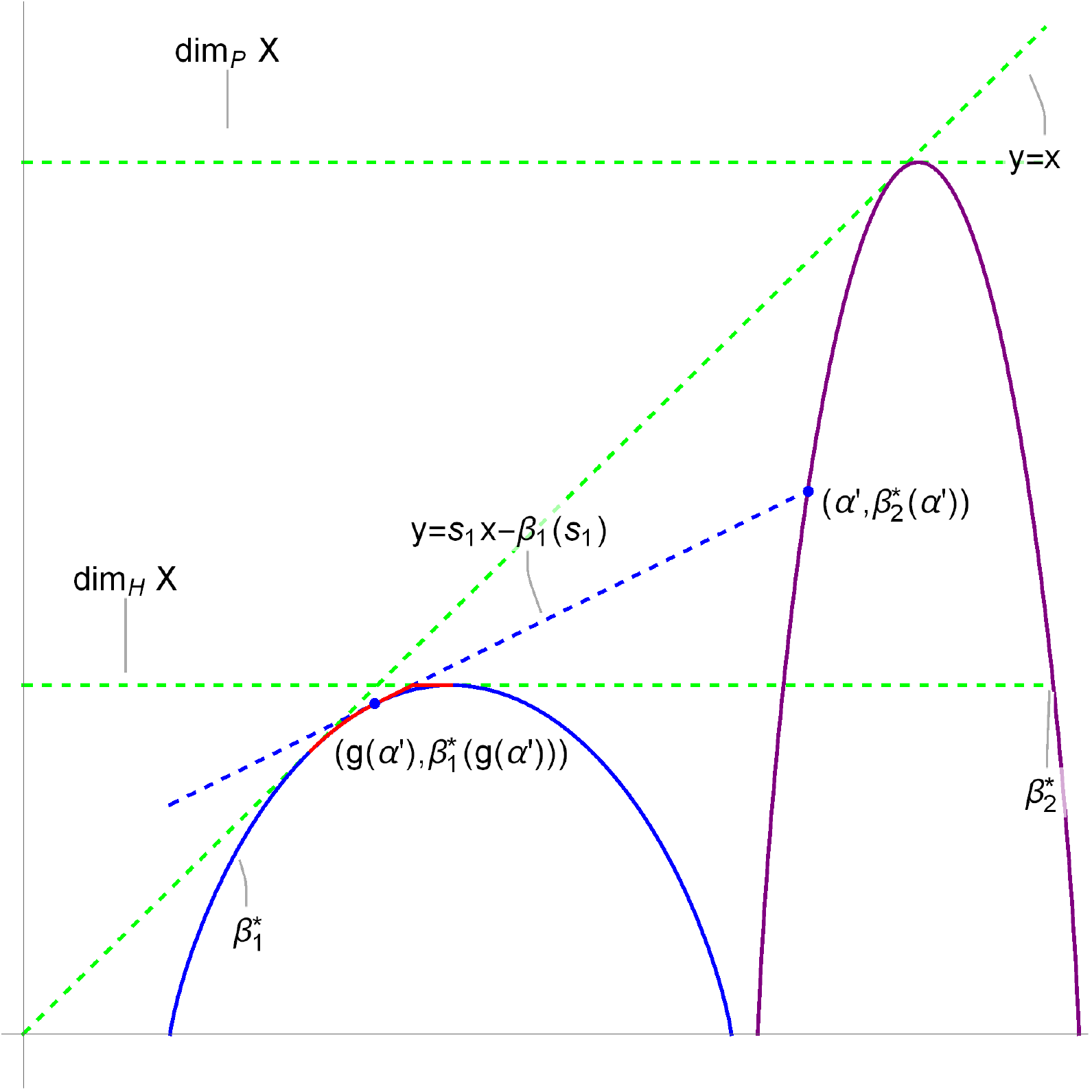}  
		\caption{$\dim_{H}(E(\mu,\cdot,\alpha'))$}  
		\label{fig:Thm-1}  
	\end{minipage}
	\begin{minipage}[t]{0.5\linewidth}  
		\centering  
		\includegraphics[width=2.8in]{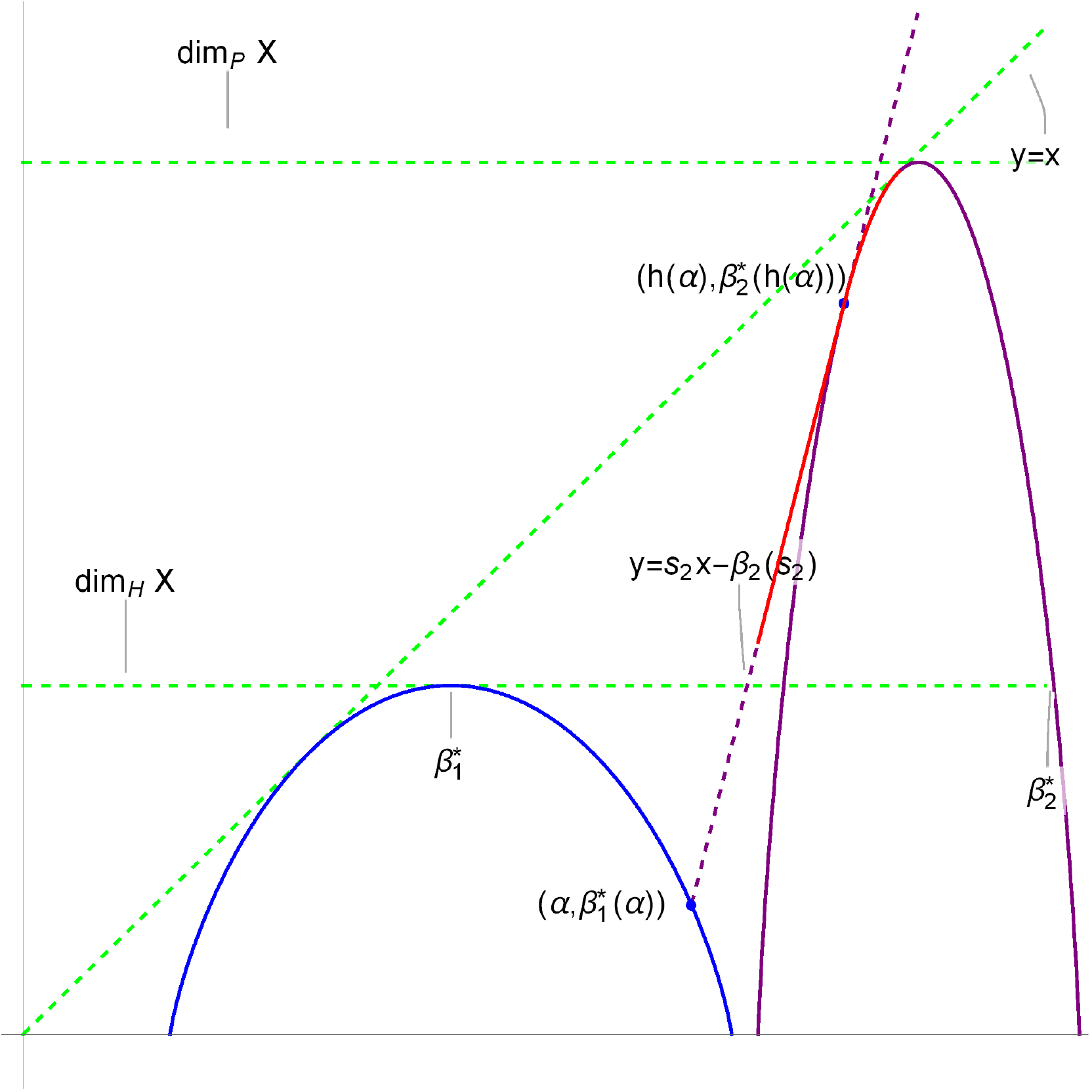}  
			\caption{$\dim_{P}(E(\mu,\alpha,\cdot))$}  
		\label{fig:Thm-2}  
	\end{minipage}  
\end{figure}

Let us give some remarks of this paper.
\begin{enumerate}

	\item We refer the reader to \cite{Yuan2017}, where the author deals with the random weak Gibbs measures. Under some natural conditions on transitivity and conformality, almost surely, it gives a full description of the multifractal analysis. But, unfortunately, it could not give any information for a sample in the exceptional set with measure $0$. A crucial problem is to deal with such samples. Is this still true? Is the level set with local dimension always nonempty? What can their multifractal behavior be? Our model provides a non-trivial sample with a full description of the multifractal analysis. 
    
	\item The inequality in~\eqref{ass} is not critical for the calculation of the dimension of level sets. Even if we do not have the inequality, our method can also be effective. In this sense, we can get the full results for the example in \cite{BP2013}. We focus this model to avoid discussing too many situations.
	
	\item We need to emphasize that we are dealing the level set with respect to upper and lower local dimensions under the nonexistence of the "free energy" function. The lower bounds for $\dim_H\El(\mu,\alpha)$ when $\alpha\in  [\beta_{1}'(1),\beta_{1}'(0))$ and  $\dim_P\Eu(\mu,\alpha)$ when $\alpha\in  [\beta_{2}'(+\infty),\beta_{1}'(1))$ are new phenomena and have not been obtained before. For the set $E(\mu,\alpha,\alpha')$, the upper bounds of the (Hausdorff and packing) dimensions are also new phenomena for the part if those of $\El(\mu,\alpha)$ and $\Eu(\mu,\alpha)$ could not provide the exact upper bounds.     
	  
	\item The strong law of large numbers \cite[Theorem 5.4.1]{Chung2001} plays an important role in this paper. The classical strong law of large numbers can deal with infinite sequence of independent identically distributed random variables, but in our situation they are not identically distributed.   
	       
%
\end{enumerate}

\section{A general model}
We now set a generalization of the model presented in section ~\ref{section: The Construction and main results}.
Fix a sequence of positive numbers $\{A_n\}_{n\in\mathbb{N}}$ with $A_n\geq2$ for any $n\in\mathbb{N}$.

Now we consider the following construction:
\begin{enumerate}
	\item [Step 1:] Let $I=[0,1]$,
	for $n=1$, define:
	$I_{0}=[0,1/A_1]\subset I,$ and $I_{1}=[1-1/A_{1},1]\subset I.$
	\item [Step 2:] For $n\in\mathbb{N}$ with $n>1$, we assume that for all $w\in \mathcal{A}^{n-1}$, the set $I_{w}$ have been defined. Let $x_w$ be the left endpoint of $I_{w}$. Define $I_{w\ast 0}=[x_w,x_w+\frac{|I_{w}|}{A_n}]$ and $I_{w\ast 1}=[x_w+|I_{w}|-\frac{|I_{w}|}{A_n},x_w+|I_{w}|]$.	
	\item [Step 3:] Define $Y=\cap_{n\in\mathbb{N}}\cup_{w\in\mathcal{A}^{n}}I_{w}$.
\end{enumerate}

Then, fix $a,b\in\mathbb{R}$ such that $0<a\leq b<1$ and a sequence of positive numbers $\{p_n\}_{n\in\mathbb{N}}$ with $a\leq p_n\leq b$ for any $n\in\mathbb{N}$.

Now we distribute a measure on $Y$:
\begin{enumerate}
	\item [Step 1:] $\nu(I)=1$.
	For $n=1$, define: $\nu(I_0)=p_1$ and $\nu(I_1)=1-p_1$.
	\item [Step 2:] For $n\in\mathbb{N}$ with $n>1$, we assume that $\nu(I_{\sigma})$ have been defined for all $\sigma\in \mathcal{A}^{n-1}$.
	Define $\nu(I_{\sigma\ast0})=p_{n}\nu(I_{\sigma})$ and $\nu(I_{\sigma\ast 1})=(1-p_n)\nu(I_{\sigma})$.
	\item [Step 4:] We can extend $\nu$ to be a probability measure on $Y$ by measure extension theorem, then ${\rm supp}(\nu)=Y$.
\end{enumerate}

Define $d(\nu,x,n)=\frac{\log \nu(I_{n}(x))}{\log |I_{n}(x)|}$, where $I_{n}(x)$ is the cylinder $I_{w}$ with $x\in I_{w}$ and $w=w_1w_2\cdots w_{n}$ of length $n$.

Now we choose a sequence $\{p_n'\}_{n\in\mathbb{N}}$ with $0\leq p_n'\leq 1$. We consider the measure $\nu'$ constructed as $\nu$ with parameters $\{p_n'\}_{n\in\mathbb{N}}$ instead of $\{p_n\}_{n\in\mathbb{N}}$.

Now we state a version of strong law of large numbers which will be useful in this paper.
\begin{theorem}\label{strong law of large numbers}\cite[Theorem 5.4.1]{Chung2001}
	Let $\{X_n\}_{n\in\mathbb{N}}$ be a sequence of independent random variables with $\mathbb{E}(X_n)=0$ for every $n\in\mathbb{N}$, and $\{a_n\}_{n\in\mathbb{N}}$ positive and increasing to $+\infty$. If there exists a function $\phi$ such that $\phi$ is positive, even and continuous on $\mathbb{R}$ such that 
	\begin{itemize}
		\item when $|x|$ increases, $\frac{\phi(x)}{|x|}$ does not decrease and $\frac{\phi(x)}{x^2}$ does not increase,
		\item $\sum_{n}\frac{\mathbb{E}(\phi(X_n))}{\phi(a_n)}<+\infty$.
	\end{itemize}
	Then $\sum_{n}\frac{X_n}{a_n}$ converge almost everywhere.
	Furthermore, $\frac{1}{a_n}\sum_{i=1}^nX_i$ converges to $0$ as $n\to\infty$ almost everywhere. 
\end{theorem}

From theorem \ref{strong law of large numbers}, we can get the following lemma.
\begin{lemma}\label{an limit}
	\begin{itemize}
		\item [1)] For $\nu'$-almost every $x\in Y={\rm supp}(\nu)$, one has $$\liminf_{n\to\infty} d(\nu,x,n)=\liminf_{n\to \infty}\frac{\sum_{k=1}^{n}x_k}{\sum_{k=1}^{n}b_k}$$
		and
		$$\limsup_{n\to\infty} d(\nu,x,n)=\limsup_{n\to \infty}\frac{\sum_{k=1}^{n}x_k}{\sum_{k=1}^{n}b_k}.$$
		where $x_k=-p_k'\log p_k-(1-p_k')\log (1-p_k)$.
		\item [2)] For $\nu'$-almost every $x\in {\rm supp}(\nu')\subset Y$, one has $$\liminf_{n\to\infty} d(\nu',x,n)=\liminf_{n\to \infty}\frac{\sum_{k=1}^{n}y_k}{\sum_{k=1}^{n}b_k}$$
		and
		$$\limsup_{n\to\infty} d(\nu',x,n)=\limsup_{n\to \infty}\frac{\sum_{k=1}^{n}y_k}{\sum_{k=1}^{n}b_k}.$$
		where $y_k=-p_k'\log p_k'-(1-p_k')\log (1-p_k')$.
	\end{itemize}
\end{lemma}  

\begin{proof}
We just fix our attention on item 1), the next one is the same.

Define the random variable $X_{k}$ as 
$$X_k(x)=\begin{cases}
-\log p_k&\,I_{k}(x)=I_{w} \text{ with } w_k=0,\\
-\log(1-p_{k})&\,I_{k}(x)=I_{w} \text{ with } w_k=1.
\end{cases}$$
Then, for the measure $\nu'$, $\{X_k-\mathbb{E}(X_k)\}_{n\in\mathbb{N}}$ is a sequence of independent random variables with $\mathbb{E}(X_k-\mathbb{E}(X_k))=0$ 

Define $b_k=\log A_k$, we replace the $X_n$, $a_n$ and $\phi$ in theorem~\ref{strong law of large numbers} by $X_n-\mathbb{E}(X_n)$, $\sum_{k=1}^n b_k$ and the function $s\mapsto s^2$, we can conclude for $\nu'$-almost every $x\in Y$ 
$$\lim_{n\to \infty}\frac{\sum_{k=1}^{n}(X_k(x)-\mathbb{E}(X_k))}{\sum_{k=1}^{n}b_k}=0.$$
So that
$$\liminf_{n\to \infty}\frac{\sum_{k=1}^{n}X_k(x)}{\sum_{k=1}^{n}b_k}=\liminf_{n\to \infty}\frac{\sum_{k=1}^{n}\mathbb{E}(X_k)}{\sum_{k=1}^{n}b_k}\ \ \text{ and }\ \ \limsup_{n\to \infty}\frac{\sum_{k=1}^{n}X_k(x)}{\sum_{k=1}^{n}b_k}=\limsup_{n\to \infty}\frac{\sum_{k=1}^{n}\mathbb{E}(X_k)}{\sum_{k=1}^{n}b_k}.$$
Item 1) valid if we notice that $d(\nu,x,n)=\frac{\sum_{k=1}^{n}X_k(x)}{\sum_{k=1}^{n}b_k}$ and $\mathbb{E}(X_k)=x_k$.
\end{proof}

Now if we assume that $2<\underline{A}:=\inf\{A_n:n\in\mathbb{N}\}$, then we know that the strong separation condition holds. We now also assume  $\overline{A}:=\sup\{A_n:n\in\mathbb{N}\}<+\infty$. We can get the following lemma by the same method as in \cite[Corollary 4.3]{Falconer1997}.
\begin{lemma}\label{<B<}
   For any ball $B(x,r)$ with $x\in Y$ and $r>0$, we choose the smallest $n\in\mathbb{N}$ and the largest $n'\in\mathbb{N}$ such that $I_{n}(x)\subset B(x,r)$ and $(B(x,r)\cap Y)\subset I_{n'}(x)$. Then we have
   
   $\frac{r}{\overline{A}}\leq \frac{r}{A_n}\leq |I_{n}(x)|\leq 2r$ and $|I_{n'}(x)|\leq \frac{2(A_{n'+1})r}{A_{n'+1}-2}\leq \frac{2\underline{A}r}{\underline{A}-2}$ 
\end{lemma}
The proof of lemma~\ref{<B<} is obvious if we notice that $|I_{n-1}(x)|\geq r$ and the $(n'+1)$-th gap in $I_{n'}(x)$ is a subset of $B(x,r)$.

Lemma~\ref{an limit} and \ref{<B<} can deduce the following lemma.
\begin{lemma}\label{dimloc an}
	If we assume that $2<\underline{A}\leq \overline{A}<+\infty$,
	\begin{enumerate}
		\item[1)] for each point $x\in Y={\rm supp}(\nu)$,
		$$\underline{\dim}_{\rm loc}(\nu',x)=\liminf_{n\to\infty} {d(\nu',x,n)}\ \ \text{and} \ \ \overline{\dim}_{\rm loc}( \nu',x)=\limsup_{n\to\infty}  d(\nu',x,n).$$
		\item [2)] the measure $\nu'$ is exact upper and lower dimensional. 
		
	\end{enumerate}

\end{lemma}

\section{Proof of the results} 

Now we return to the model defined in section~\ref{section: The Construction and main results}. 

\subsection{Proof of Theorem~\ref{dim-loc-bound}}
In fact we just need to notice 
 $\lim_{i\to\infty}\frac{N_{i+1}}{N_i}=\infty$, lemma~\ref{an limit} and~\ref{dimloc an}. The proof is the following.

\begin{proof}
	\begin{enumerate}
	\item We just need to notice that $$-\frac{\log (1-p)}{\log A}\leq d(\nu,x,n)\leq  -\frac{\log q}{\log B},$$
	$$\liminf_{i\to\infty}d(\nu,x,N_{2i+1})\leq-\frac{\log p}{\log A}$$
	and
	$$\limsup_{i\to\infty}d(\nu,x,N_{2i})\geq -\frac{\log (1-q)}{\log B}.$$
	The results can be obtained from lemma~\ref{dimloc an}.
	\item It is from inequality \eqref{ass} and the previous item.
	\item It is direct due to the construction and lemma~\ref{an limit} and~\ref{dimloc an}.
	\item  We just need to distribute a mass replacing $p$ and $q$ by $1/2$, it is obvious that for all $x\in X$, $\underline{\dim}_{\rm loc}(\mu,x)=\frac{\log 2}{\log A}$ and $\overline{\dim}_{\rm loc}(\mu,x)=\frac{\log 2}{\log B}$. By \cite[Proposition 2.3]{Falconer1997}, the results hold.  
\end{enumerate}	
\end{proof}

\subsection{Proof of Theorem~\ref{main}}

    From the construction, since $A>B>2$ and the lengths of  cylinders $I_{w}, w\in \mathcal{A}^n$ are equal for any $n\in\mathbb{N}$, we can easily obtain the following lemma from lemma~\ref{<B<}.
    \begin{lemma}
    	\begin{equation}\label{tau-l}
    	\underline\tau(s)=\min\{\beta_1(s),\beta_2(s)\}
    	\end{equation}
    	and
    	\begin{equation}\label{tau-u}
    	\overline{\tau}(s)=\max\{\beta_1(s),\beta_2(s)\}.
    	\end{equation}	   
    \end{lemma}

    We also need to introduce the large deviations spectra (see also \cite{Barral2015}) which will be useful in our proof. 

	Let $\mu$ be a compactly supported positive and finite Borel measure on a metric space.
	For $0\leq \alpha\leq \beta\leq+\infty$, the lower and upper large deviations spectra $\underline{f}_{\mu}^{LD}$ and $\overline{f}_{\mu}^{LD}$ are given by
	\begin{equation*}\label{LD}
	\underline{f}_{\mu}^{LD}(\alpha,\beta)=\lim_{\varepsilon\to 0}\liminf_{r\rightarrow 0}\frac{\log \sup\#\{i:r^{\beta+\varepsilon}\leq\mu(B(x_i,r)\leq r^{\alpha-\varepsilon})\}}{-\log r},
	\end{equation*}
	\begin{equation*}\label{LDu}
	\overline{f}_{\mu}^{LD}(\alpha,\beta)=\lim_{\varepsilon\to 0}\limsup_{r\rightarrow 0}\frac{\log \sup\#\{i:r^{\beta+\varepsilon}\leq\mu(B(x_i,r)\leq r^{\alpha-\varepsilon})\}}{-\log r},
	\end{equation*}
	where the supremum is taken over all families of disjoint closed balls $B_i=B(x_i,r)$ of radius $r$ with centers $x_i$ in $\supp(\mu)$. 

Now we turn to consider the level sets and prove theorem~\ref{main}.

For any $\alpha\in [\beta_1'(+\infty),\beta_1'(-\infty)]$, if there exists $s\in\mathbb{R}$ such that $\beta'_1(s)=\alpha$, we can define the number $A^{\beta_1(s)}\cdot p^s$, otherwise, $\alpha\in\{\beta_1'(+\infty),\beta_1'(-\infty)\}$, we write $A^{\beta_1(+\infty)}\cdot p^{+\infty}:=\lim_{s\to +\infty}A^{\beta_1(s)}\cdot p^s$ and $A^{\beta_1(-\infty)}\cdot p^{-\infty}:=\lim_{s\to -\infty}A^{\beta_1(s)}\cdot p^s$. 
Also, we use the same notation for $\beta_2$.

\begin{proof}
	\begin{enumerate}
		\item It is obvious because of theorem~\ref{dim-loc-bound}.
        \item Now we consider the level set $\El(\mu,\alpha))$.
        \begin{itemize}
        	\item \textbf{Upper bound for $\dim_{H}(\underline{E}(\mu,\alpha))$ and $\dim_{P}(\underline{E}(\mu,\alpha))$:}\\ 
        One obvious upper bound is $\dim_{H}(\underline{E}(\mu,\alpha))\leq \dim_{H}X=\frac{\log 2}{\log A}$ and  $\dim_{P}(\underline{E}(\mu,\alpha))\leq \dim_{P}X=\frac{\log 2}{\log B}$.
        But it is not sharp for the Hausdorff dimension, so we need the following estimation.
        \begin{itemize}
        	\item For $\alpha\in [\beta_1'(+\infty),\beta_1'(0)]$, noticing the result in \cite[Proposition 1.3 \& inequality (1.5)]{Barral2015}, we know that 
        	$$\dim_{H}\underline{E}(\mu,\alpha)\leq \overline{f}_{\mu}^{LD}(\alpha,\alpha)\leq\underline\tau^{*}(\alpha)=\begin{cases}
        	\beta_1^{*}(\alpha)&\,\alpha\in [\beta_1'(+\infty),\beta_1'(1)],\\
        	\alpha&\,\alpha\in [\beta_1'(1),\beta_1'(0)].
        	\end{cases}$$
            
        	\item For $\alpha\in [\beta_1'(0),\beta_1'(-\infty)]$. Noticing the result in \cite[Proposition 1.3]{Barral2015}, $\dim_{H}\underline{E}(\mu,\alpha)\leq \underline{f}_{\mu}^{LD}(\alpha,+\infty)=:f_1(\alpha)$.
        	
        	We just need to prove that $\beta_1(s)\leq (f_1)^{*}(s)$ for $s\in (-\infty,0]$.
        	If so,
        	\begin{eqnarray*}
        		\beta_1^{*}(\alpha)&=&\inf\{s\alpha-\beta_{1}(s)\}\\
        		&=&\inf_{s\leq 0}\{s\alpha-\beta_{1}(s)\}\\
        		&\geq&\inf_{s\leq 0}\{s\alpha-(f_1)^{*}(s)\}\\
        		&\geq&\inf_{s\leq 0}\{s\alpha-(s\alpha-f_1(\alpha))\}\\
        		&=&f_1(\alpha),
        	\end{eqnarray*} 
        	and then  $\dim_{H}\underline{E}(\mu,\alpha)\leq \beta_1^{*}(\alpha).$
        	
        	Now turn to prove $\beta_1(s)\leq (f_1)^{*}(s)$ for $s\in (-\infty,0]$. We will borrow the main idea in \cite[Section 5.3]{Barral2015}, but make it clear that we are dealing with $s\leq 0$.
        	If $\{B(x_i,r)\}$ is a packing of ${\rm supp}(\mu)$ by disjoint balls, we have for any $t\in\mathbb{R}$,
        	$$\sum_{i}\mu(B(x_i,r))^{s}\geq (\# \{i:\mu(B(x_i,r))\leq r^{t-\epsilon} \})r^{s(t-\epsilon)}.$$ 
        	Taking the supremum over the packings, dividing by $\log r$, taking the lim sup as $r\to 0^{+}$ and then the limit $\epsilon\to 0^{+}$ yields $\beta_{1}(s)=\overline{\tau}(s)\leq st-f_1(t)$ for all $t\in\mathbb{R}$, that is 
        	$\beta_{1}(s)\leq (f_1)^{*}(s)$
        \end{itemize}
    
        \item \textbf{Lower bound for $\dim_{H}(\underline{E}(\mu,\alpha))$ and $\dim_{P}(\underline{E}(\mu,\alpha))$:}
        
        For any $\alpha\in [\beta_1'(+\infty),\beta_1'(-\infty)]$, there exists $s\in\mathbb{R}\cup\{+\infty,-\infty\}$ such that $\beta_2'(s)=\alpha$.
        Choose  $$p_n'=\begin{cases}
        A^{\beta_1(s)}\cdot p^s&\, N_{2i}<n\leq N_{2i+1} \text{ for some } i\in\mathbb{N},\\
        1/2&\, N_{2i+1}<n\leq N_{2i+2}\text{ for some } i\in\mathbb{N}.\end{cases}$$
        to define $\mu'$. 
        
        Noticing 
        $$\frac{-p'\log p-(1-p')\log(1-p)}{\log A}= \beta_1'(s)=\alpha,$$
        where $p'=:A^{\beta_1(s)}\cdot p^s$.
        From lemma~\ref{an limit} and~\ref{dimloc an} and the choice of $s$, we get that $\mu'(\underline{E}(\mu,\alpha))=1$ and for $\mu'$-almost every $x\in X$ one has $\underline{\dim}_{\rm loc}(\mu',x)=\frac{-p'\log p'-(1-p')\log(1-p')}{\log A}=s\alpha-\beta_{1}(s)= \beta_{1}^*(\alpha)$ and $\overline{\dim}_{\rm loc}(\mu',x)=\frac{\log 2}{\log B}$.
        This gives that $$\dim_{H}\underline{E}(\mu,\alpha)\geq\beta_1^{*}(\alpha) $$ and $$\dim_{P}\underline{E}(\mu,\alpha)\geq\frac{\log 2}{\log B}.$$

        But it is not sharp for $\dim_{H}\underline{E}(\mu,\alpha)$ when $\alpha\in  (\beta_1'(1),\beta_1'(0))$. 
        Now take $\alpha_1=\frac{-p\log p-(1-p)\log(1-p)}{\log A}$ and $\alpha_2=\frac{-q\log q-(1-q)\log(1-q)}{\log B}$, define $$p_n'=\begin{cases}
        p&\, N_{2i}<n\leq N_{2i}'\text{ for some } i\in\mathbb{N},\\
        1/2&\, N_{2i}'<n\leq N_{2i+1}\text{ for some } i\in\mathbb{N}\\
        q &\, N_{2i+1}<n\leq N_{2i+2} \text{ for some } i\in\mathbb{N}.
        \end{cases}$$
        where $N_{2i}'=\min\{N_{2i+1}, \lfloor\frac{(\alpha_2-\alpha)(\sum_{k=0}^{i-1}(N_{2k+2}-N_{2k+1}))\log B}{(\alpha-\alpha_1)\log A}\rfloor-\sum_{k=0}^{i-1}(N_{2k}'-N_{2k})\}$. This also gives a measure $\mu'$.
        
        We use the same notation as in lemma~\ref{an limit} and turn to analysis $\frac{\sum_{k=1}^{n}x_k}{\sum_{k=1}^{n}b_k}$. 
        For any $n\in \mathbb{N}$, we assume that $l_1(n)=\#\{k\leq n:p_k'=p\}, l_2(n)=\#\{k\leq n:p_k'=1/2\}$ and $l_3(n)=\#\{k\leq n:p_k'=q\}$ 
        then $l_1(n)+l_2(n)+l_3(n)=n$. From the choice of $N_{2i}'$, we can get 
        $l_1(n)\leq \frac{(\alpha_2-\alpha)l_3(n)\log B}{(\alpha-\alpha_1)\log A}$,
        which yields
        \begin{equation}\label{l123-ineq}
        \alpha_1l_1(n)\log A+\alpha_2l_3(n)\log B\geq \alpha(l_1(n)\log A+l_3(n)\log B).
        \end{equation}
        
        For $n$ large enough, on the one hand, 
        \begin{eqnarray*}
        &&\frac{\sum_{k=1}^{n}x_k}{\sum_{k=1}^{n}b_k}\\
        &=&\frac{\alpha_1l_1(n)\log A-\frac{(\log q+\log(1-q))l_2(n)}{2}+\alpha_2l_3(n)\log B}{l_1(n)\log A+(l_2(n)+l_3(n))\log B}\\
        &\geq& \frac{(\alpha_1l_1(n)\log A+\alpha_2l_3(n)\log B)-\frac{(\log q+\log(1-q))l_2(n)}{2}}{(l_1(n)\log A+l_3(n)\log B)+l_2(n)\log B},\\
        \end{eqnarray*}
        where the inequality from ~\eqref{l123-ineq}.
        Now we can get 
        \begin{equation}\label{A-n-l}
        \frac{\sum_{k=1}^{n}x_k}{\sum_{k=1}^{n}b_k}\geq \alpha,
        \end{equation}
        since the inequlity in \eqref{ass}.
        
        On the other hand, for $n=N_{2i+1}'$, we notice that 
        $$l_1(n)\geq \frac{(\alpha_2-\alpha)l_3(n)\log B}{(\alpha-\alpha_1)\log A}-1,$$
        then
       	\begin{eqnarray*}
       		&&\frac{\sum_{k=1}^{n}x_k}{\sum_{k=1}^{n}b_k}\\
       		&=&\frac{\alpha_1l_1(n)\log A-\frac{(\log q+\log(1-q))l_2(n)}{2}+\alpha_2l_3(n)\log B}{l_1(n)\log A+(l_2(n)+l_3(n))\log B}\\
       		&=& \frac{(\alpha_1(l_1(n)+1)\log A+\alpha_2l_3(n)\log B)-\alpha_1\log A-\frac{(\log q+\log(1-q))l_2(n)}{2}}{(l_1(n)\log A+l_3(n)\log B)+l_2(n)\log B}\\
       		&\leq&\frac{\alpha((l_1(n)+1)\log A+l_3(n)\log B)-\alpha_1\log A-\frac{(\log q+\log(1-q))l_2(n)}{2}}{((l_1(n)+1)\log A+l_3(n)\log B)-\log A+l_2(n)\log B}
       	\end{eqnarray*} 
       The assumption $\lim_{i\to \infty}\frac{N_{i+1}}{N_{i}}=\infty$ yields $\lim_{i\to \infty}\frac{l_2(N_{2i+1}')\log B}{l_1(N_{2i+1}')\log A+l_3(N_{2i+1}')\log B}=0$, and then
       $\limsup_{i\to\infty}\frac{\sum_{k=1}^{N_{2i+1}'}x_k}{\sum_{k=1}^{N_{2i+1}'}b_k}\leq\alpha$. Noticing inequality~\eqref{A-n-l}, we have
       $$\liminf_{n\to\infty}\frac{\sum_{k=1}^{n}x_k}{\sum_{k=1}^{n}b_k}=\alpha.$$
       Hence $\mu'(\underline{E}(\mu,\alpha))=1$ by lemma~\ref{an limit} and ~\ref{dimloc an}.
       A similar method yields $$\underline{\dim}_{\rm loc}(\mu',x)=\min\{\alpha,\frac{\log 2}{\log A}\}$$ for $\mu'$-almost every $x\in X$.
        
       We now can conclude that $\dim_{H}\underline{E}(\mu,\alpha)\geq \min\{\alpha,\frac{\log 2}{\log A}\}$ from \cite[item (a) in Proposition 2.3]{Falconer1997}.
    \end{itemize}

      \item We now turn to consider $\Eu(\mu,\alpha)$.
      \begin{itemize}
      	\item \textbf{Upper bound for $\dim_H(\overline{E}(\mu,\alpha))$:}
      	The one obvious upper bound is $\dim_H X=\frac{\log 2}{\log A}$ and the other is the following.
      	
      \begin{itemize}
      \item For $\alpha\in [\beta_2'(1),\beta_2'(-\infty)]$. noticing the result in \cite[Proposition 1.3 \& inequality (1.5)]{Barral2015}, we know that
      $$\dim_{H}\overline{E}(\mu,\alpha)\leq \overline{f}_{\mu}^{LD}(\alpha,\alpha)\leq \underline\tau^{*}(\alpha)=\beta_2^{*}(\alpha).$$
      \item For $\alpha\in [\beta_2'(+\infty),\beta_2'(1))$, the proof is the same as the proof of the upper bound for $\dim_{H}(\underline{E}(\mu,\alpha))$ with $\alpha\in [\beta_1'(0),\beta_1'(-\infty)]$, so we just list the key points. 
       
      Notice that $\dim_{H}\overline{E}(\mu,\alpha)\leq \underline{f}_{\mu}^{LD}(0,\alpha):=f_2(\alpha)$.
      Since $\beta_2^{*}(\alpha)=\inf_{s\geq 1}\{s\alpha-\beta_{2}(s)\}$, we just need to prove $\beta_2(s)\leq (f_2)^{*}(s)$ for $s\in [1,+\infty]$.
      If $\{B(x_i,r)\}$ is a packing of ${\rm supp}(\mu)$ by disjoint balls, we have for any $t\in\mathbb{R}$,
      $$\sum_{i}\mu(B(x_i,r))^{s}\geq (\# \{i:\mu(B(x_i,r))\geq r^{t+\epsilon} \})r^{s(t+\epsilon)}.$$
      and $\beta_{2}(s)=\overline{\tau}(s)$ for $s\in [1,+\infty]$.
      \end{itemize}
  
     \item \textbf{The uppper bound for $\dim_P(\overline{E}(\mu,\alpha))$}:
     
     Noticing \cite[Proposition 1.3 \& inequality (1.5)]{Barral2015}, we have
     \begin{equation*}\label{key}
     \dim_P(\overline{E}(\mu,\alpha))\leq\sup\{\tau^*(\alpha'): \alpha'\leq \alpha\}=\begin{cases}
     \alpha& \alpha\in[\beta_{2}'(+\infty),\beta_{2}'(1)),\\
     \beta_2^*(\alpha)& \,\alpha\in [\beta_{2}'(1),\beta_{2}'(0))\\
     \frac{\log 2}{\log B}& \,\alpha\in [\beta_{2}'(0),\beta_{2}'(-\infty)].
     \end{cases}
     \end{equation*}
     
  \item\textbf{ Lower bound for  $\dim_H(\overline{E}(\mu,\alpha))$:}\\
  For $\alpha\in [\beta_2'(+\infty),\beta_2'(-\infty)]$, there exists $s\in \mathbb{R}\cup\{+\infty,-\infty\}$ such that $\beta_2'(s)=\alpha$, choose  $$p_n'=\begin{cases}
  1/2&\, N_{2i}<n\leq N_{2i+1} \text{ for some } i\in\mathbb{N}.\\
  B^{\beta_2(s)}\cdot q^s&\, N_{2i+1}<n\leq N_{2i+2}\text{ for some } i\in\mathbb{N}\end{cases}.$$
  Now we have defined a measure $\mu'$. 
  
  Noticing that 
  $$\frac{-q'\log q-(1-q')\log(1-q)}{\log B}= \beta_2'(s)=\alpha,$$
  where $q'=B^{\beta_2(s)}\cdot q^s$.
  
  From lemma~\ref{dimloc an} and~\ref{an limit} and the choice of $s$, we get that $\mu'(\underline{E}(\mu,\alpha))=1$ and for $\mu'$-almost every $x\in X$ one has $\underline{\dim}_{\rm loc}(\mu',x)=\min\{\frac{\log 2}{\log A},\beta_2^*(\alpha)\}$.
  This gives $$\dim_{H}\underline{E}(\mu,\alpha)\geq\min\{\frac{\log 2}{\log A},\beta_2^*(\alpha)\}.$$
  
  \item \textbf{Lower bound for  $\dim_P(\overline{E}(\mu,\alpha))$:}
  \begin{itemize}
  	\item For $\alpha\in [\beta_2'(+\infty),\beta_2'(1))$, the proof is similar with the proof of the lower bound for $\dim_{H}(\underline{E}(\mu,\alpha))$ for $\alpha\in (\beta_2'(1),\beta_2'(0))$, so we again just list the key points.
  	 
  	First, there exists $s\in(1,+\infty]$ such that $\beta_2'(s)=\alpha$.
  	
  	Second, take $\alpha_1=\frac{-p\log p-(1-p)\log(1-p)}{\log A}$ and $\alpha_2=\frac{-q\log q-(1-q)\log(1-q)}{\log B}$.
  	Define $$p_n'=\begin{cases}
  	p&\, N_{2i}<n\leq N_{2i+1} \text{ for some } i\in\mathbb{N},\\
  	q&\, N_{2i+1}<n\leq N_{2i+1}'\text{ for some } i\in\mathbb{N},\\
  	q^s B^{\beta_2(s)}&\, N_{2i+1}'<n\leq N_{2i+2}\text{ for some } i\in\mathbb{N}.
  	\end{cases}$$
  	where $N_{2i+1}'=\min\{N_{2i+2}, \lfloor\frac{(\alpha_2-\alpha)(\sum_{k=0}^{i}(N_{2k+1}-N_{2k}))\log A}{(\alpha-\alpha_1)\log B}\rfloor-\sum_{k=0}^{i-1}(N_{2k+1}'-N_{2k+1})\}$. This also gives a measure $\mu'$.
  	We turn to analysis $\frac{\sum_{k=1}^{n}x_k}{\sum_{k=1}^{n}b_k}$. 
  	For any $n\in \mathbb{N}$, we assume that $l_1(n)=\#\{k\leq n:p_k'=p\},l_2(n)=\#\{k\leq n:p_k'=q\}$ and $l_3(n)=\#\{k\leq n:p_k'=q^s B^{\beta_2(s)}\}$ 
  	then $l_1(n)+l_2(n)+l_3(n)=n$.
  	In the same way, for $n$ large enough, we have
  	\begin{equation*}
   \frac{\sum_{k=1}^{n}x_k}{\sum_{k=1}^{n}b_k}\leq \alpha
  	\end{equation*} 
  	and also we pay our attention to $n=N_{2i+1}'$
      	\begin{equation*}
    \frac{\sum_{k=1}^{n}x_k}{\sum_{k=1}^{n}b_k}\geq \frac{\alpha(l_1(n)\log A+(l_2(n)+1)\log B)-\alpha_2\log B+\alpha l_3(n)\log B}{(l_1(n)\log A+(l_2(n)+1)\log B)-\log B+l_3(n)\log B}.
    \end{equation*}
     Then,
     $\limsup_{i\to\infty}\frac{\sum_{k=1}^{N_{2i}'}x_k}{\sum_{k=1}^{N_{2i}'}b_k}=\alpha$.
     Hence $\mu'(\overline{E}(\mu,\alpha))=1$ by lemma~\ref{an limit} and ~\ref{dimloc an}.
  	A similar method yields $\overline{\dim}_{\rm loc}(\mu',x)=\max\{\alpha,\beta_2^{*}(\alpha)\}$ for $\mu'$-almost every $x\in X$. 
  	We now get that $\dim_{P}\overline{E}(\mu,\alpha)\geq \max\{\alpha,\beta_2^{*}(\alpha)\}$.
  	
  	\item $\alpha\in (\beta_2'(1),\beta_2'(-\infty)]$. There exists $s\in(1,+\infty]$ such that $\beta_2'(s)=\alpha$. Let $\beta_2^*(\alpha_0)=\max\{\beta_2^*(\alpha'):\alpha'\leq \alpha\},$ we can assume that $\alpha_0=\beta_2'(s_0)$.
  	Define $$p_n'=\begin{cases}
  	1/2&\, N_{2i-1}<n\leq N_{2i} \text{ for some } i\in\mathbb{N},\\
  	B^{\beta_2(s_0)}p^{s_0}&\, N_{2i}<n\leq N_{2i}'\text{ for some } i\in\mathbb{N},\\
  	B^{\beta_2(s)}p^{s}&\, N_{2i}'<n\leq N_{2i+1}\text{ for some } i\in\mathbb{N}.
  	\end{cases}$$
  	Where $\{N_{2i}'\}_{i\in\mathbb{N}}$ is a sequence of numbers such that
  	$N_{2i}<N_{2i}'\leq N_{2i+1}$, $\lim_{i\to\infty}\frac{N_{2i}'}{N_{2i}}=+\infty$ and 
  	$\lim_{i\to\infty}\frac{N_{2i+1}}{N_{2i}'}=+\infty$.
  	This can be done since $\lim_{i\to \infty}\frac{N_{i+1}}{N_{i}}=+\infty$.
  	Now we can check that 
  	$\limsup_{n\to\infty}\frac{\sum_{k=1}^{n}x_k}{\sum_{k=1}^{n}b_k}=\alpha$
  	since $\lim_{i\to\infty}\frac{N_{2i+1}}{N_{2i}'}=+\infty$ and $\alpha\geq \alpha_0\geq \frac{\log 2}{\log A}$.
  	 Hence $\mu'(\overline{E}(\mu,\alpha))=1$ by lemma~\ref{an limit} and ~\ref{dimloc an}.
  	 Also, noticing that $\lim_{i\to\infty}\frac{N_{2i}'}{N_{2i}}=+\infty$ and $\beta_2^*(\alpha_0)=\max\{\beta_2^*(\alpha'):\alpha'\leq \alpha\}\geq \max\{\beta_2^*(\alpha),\frac{\log 2}{\log A}\}$, a similar method yields $\overline{\dim}_{\rm loc}(\mu',x)=\beta_2^*(\alpha_0)$ for $\mu'$-almost every $x\in X$.
  	 We now get that $\dim_{P}\overline{E}(\mu,\alpha)\geq\max\{\beta_2^*(\alpha'):\alpha'\leq \alpha\}$.   	 	
  \end{itemize}
\end{itemize}
\end{enumerate}
\end{proof}

\begin{remark}
	\begin{enumerate}
		\item 	In fact, the upper bound can also be obtained if we use the result in~\cite[Proposition 2.5 \& 2.6]{Olsen1995}. (In our notation, the lower and upper $L^q$-spectrum plays the same roles as $-B$ and $-b$, see \cite{BBP2003} for details, but we need to recall the functions $b,B$ and show their relationships with the functions $\overline{\tau},\underline\tau$.) 
		\item  For the Hausdorff dimension of the level sets $\El(\mu,\alpha)$, the lower bound can be easily  obtained by using the auxiliary measure $\mu'$ as the classical method when $\alpha\in [\beta_{1}'(+\infty),\beta_{1}'(1)]\cup [\beta_{1}'(0),\beta_{1}'(-\infty)]$ . But it is a trouble for the part $[\beta_{1}'(1),\beta_{1}'(0)]]$. This part is surprising to the author when it turns out to be $\min\{\alpha,\dim_H X\}$ which is related to two points $-\frac{q\log q+(1-q)\log(1-q)}{\log A}$ and $-\frac{p\log p+(1-p)\log(1-p)}{\log B}$.
		Until now the upper bound is easy for the part $\alpha\in [\beta_1'(+\infty),\beta_1'(0)]$, since $\overline{f}_{\mu}^{LD}(\alpha,\alpha)$ (also $\tau^*_\mu(\alpha)=\beta_1^*(\alpha)$) gives the sharp bound, but it is not easy for $[\beta_1'(0),\beta_1'(-\infty)]$. In this part the Legendre transform $\tau^*$ is not the sharp upper bound, so we turn to the lower large derivation spectrum $\underline{f}_{\mu}^{LD}(\alpha,+\infty)$ and prove that it coincides to $\overline{\tau}^*=\beta_1^*$ in this part.    
		
		For packing dimension, it is equal to that of the whole space, since we almost do not disturb the part when $N_{2i+1}<n\leq N_{2i+2}$ for $\El(\mu,\alpha)$.
		\item For the level sets $\Eu(\mu,\alpha)$, there are some dualities as for $\El(\mu,\alpha)$.
		
		The Hausdorff dimension is from a traditional method except for the upper bound when $\alpha\in[\beta_2'(+\infty),\beta_2'(1))$ where we borrow $\underline{f}_{\mu}^{LD}(0,\alpha)$. 
		
		The difficulty comes from the lower bound of $\dim_{P} \Eu(\mu,\alpha)$ for the part $[\beta'_2(+\infty),\beta_2'(1)]$, where we again deal with the linear part by the two points $-\frac{p\log p+(1-p)\log(1-p)}{\log A}$ and $-\frac{q\log q+(1-q)\log(1-q)}{\log B}$.
	\end{enumerate}
	
\end{remark}

\subsection{Proof of Theorem~\ref{main2}}
The method to construct a corresponding auxiliary measure can get the sharp lower bound. A clear upper bound can be obtained if we notice that $E(\mu,\alpha,\alpha')=\El(\mu,\alpha)\cap\Eu(\mu,\alpha')$. But it is not sharp for the Hausdorff dimension when $(\alpha,\alpha')\in [\beta'_1(1),\beta'_1(0)]\times[\beta'_2(-\infty),\beta'_2(1))$ and for the packing dimension when $(\alpha,\alpha')\in (\beta'_1(1),\beta'_1(-\infty)]\times[\beta'_2(+\infty),\beta'_2(1))$. To obtain the upper bound in these situations we need to use \cite[Proposition 2.3]{Falconer1997}. 

\begin{proof}   
\begin{enumerate}
\item  Now we consider the Hausdorff dimension for $E(\mu,\alpha,\alpha')$. 
\begin{itemize}
	\item $(\alpha,\alpha')\notin [\beta'_1(1),\beta'_1(0)]\times[\beta'_2(+\infty),\beta'_2(1))$\\
	The upper bound can be obtained easily since $E(\mu,\alpha,\alpha')=\El(\mu,\alpha)\cap\Eu(\mu,\alpha')$.\\
	Let us focus our attention to the lower bound.\\
	If $\alpha\notin[\beta'_1(1),\beta'_1(0)]$, the lower bound can be obtained by the following auxiliary measure $\mu'$ defined thought $p_n'$ with
	$$p_n'=\begin{cases}
	A^{\beta_1(s_1)}p^{s_1}&\, N_{2i}<n\leq N_{2i+1}\text{ for some } i\in\mathbb{N},\\
	B^{\beta_2(s_2)}q^{s_2} &\, N_{2i+1}<n\leq N_{2i+2} \text{ for some } i\in\mathbb{N},
	\end{cases}$$
	where $\beta'_1(s_1)=\alpha,\beta'_2(s_2)=\alpha'$.
	It is easy to show that $\mu'(E(\mu,\alpha,\alpha'))=1$ and $\dim_H \mu'\geq \min\{\beta_1^*(\alpha),\beta_2^*(\alpha')\}=\min\{\dim_H(\El(\mu,\alpha)),\dim_H(\Eu(\mu,\alpha'))\}$.
	The result yields since $\dim_HE(\mu,\alpha,\alpha')\geq \dim_H \mu'$.\\
	If $\alpha\in[\beta'_1(1),\beta'_1(0)]$ and $\alpha'\notin[\beta'_2(+\infty),\beta'_2(1))$, choose $s_2\in[-\infty,1]$ such that $\beta'_2(s_2)=\alpha'$,
	let
	$$p_n'=\begin{cases}
	p&\, N_{2i}<n\leq N_{2i}'\text{ for some } i\in\mathbb{N},\\
	1/2&\, N_{2i}'<n\leq N_{2i+1}\text{ for some } i\in\mathbb{N},\\
	B^{\beta_2(s_2)}q^{s_2} &\, N_{2i+1}<n\leq N_{2i+1}' \text{ for some } i\in\mathbb{N},\\
	q &\, N_{2i+1}'<n\leq N_{2i+2} \text{ for some } i\in\mathbb{N},
	\end{cases}$$
as a similar way in the lower bound for $\dim_H \El(\mu,\alpha)$ (Also need to choose proper $\{N_i'\}_{i\in\mathbb{N}}$, but it is almost the same as in the proof of theorem~\ref{main}, so we omit the details). Then we can obtain the lower bound. 
	 
	\item If $\alpha'\in[\beta'_2(+\infty),\beta_{2}'(1))$ with $g(\alpha')\leq\beta_1'(0)$ and $\alpha\in[g(\alpha'),\beta_1'(0))$. Recall that we take the tangent to the graph $\beta_1^*$ passing through the point $(\alpha',\beta_2^*(\alpha'))$ and denote the point of tangency by $(g(\alpha'),\beta_1^*(g(\alpha')))=:(\alpha_1,\beta_1^*(\alpha_1))$.
	Then take $s_1\in[0,1], s_2\in(1,+\infty)$ such that $\beta_1'(s_1)=\alpha_1$ and $\beta_2'(s_2)=\alpha'$.
	
	Choose proper $\{N_{2i}'\}_{i\in\mathbb{N}}$, let
	 $$p_n'=\begin{cases}
	 A^{\beta_1(s_1)}p^{s_1}&\, N_{2i}<n\leq N_{2i}'\text{ for some } i\in\mathbb{N},\\
	 1/2&\, N_{2i}'<n\leq N_{2i+1}\text{ for some } i\in\mathbb{N},\\
	 B^{\beta_2(s_2)}q^{s_2} &\, N_{2i+1}<n\leq N_{2i+2} \text{ for some } i\in\mathbb{N},
	 \end{cases}$$
	 We can obtain that $\dim_H E(\mu,\alpha,\alpha')\geq\min\{\frac{\log 2}{\log A},s_1\alpha-\beta_1(s_1)\}$ in the same way as in the proof of theorem~\ref{main}.\\
	 Now we turn to the upper bound, a crucial observation is $(\beta_i^*)'(\beta_i'(s))=s$ for any $s\in\mathbb{R}$ and $i=1,2$. It is easily obtained by differentiation of a composition function. Let 
	$$p_n'=\begin{cases}
	A^{\beta_1(s_1)}p^{s_1}&\, N_{2i}<n\leq N_{2i+1}\text{ for some } i\in\mathbb{N},\\
	B^{\beta_2(s_2)}q^{s_2} &\, N_{2i+1}<n\leq N_{2i+2} \text{ for some } i\in\mathbb{N}.
	\end{cases}$$
	This defines a measure denoted by $\mu'$. Now we want to show that for any $x\in E(\mu,\alpha,\alpha')$ we have $\dl(\mu',x)\leq s_1\alpha-\beta_1(s_1)$. If so, the upper bound is obtained from \cite[Proposition 2.3 (b)]{Falconer1997} and $\dim_H X= \frac{\log 2}{\log A}$.
	
	Fix $x\in E(\mu,\alpha,\alpha')$ that is for any $\epsilon>0$, there exists $N\in\mathbb{N}$ such that for any $n\geq N$,
	$d(\mu,x,n)\leq \alpha'+\epsilon$, also for any $N'\in\mathbb{N}$, there exists $n'\geq N'$ with $d(\mu,x,n')\leq \alpha+\epsilon$.

	Choose $i_0$ large enough with $N_{2i_0}>N$ such that for any $i\geq i_0$ one has
	 \begin{equation}\label{control mu}
	 1\leq\frac{\log \mu(I_{N_{2i}}(x))}{\log \mu(I_{N_{2i}}(x))-\log \mu(I_{N_{2i-1}}(x))}\leq 1+\epsilon
	 \end{equation}
	 \begin{equation}\label{control mu'}
	 1\leq\frac{\log \mu'(I_{N_{2i}}(x))}{\log \mu'(I_{N_{2i}}(x))-\log \mu'(I_{N_{2i-1}}(x))}\leq 1+\epsilon
	 \end{equation}
	 \begin{equation}\label{control length}
	 1\leq\frac{\log |I_{N_{2i}}(x)|}{\log |I_{N_{2i}}(x)|-\log |I_{N_{2i-1}}(x)|}\leq 1+\epsilon
	 \end{equation}
	
	For any $n>N_{2i_0}$ such that $d(\mu,x,n)\leq \alpha+\epsilon$, choose the largest $i$ with $N_{2i}\leq n$, one has $\alpha_2:=d(\mu,x,N_{2i})\leq \alpha'+\epsilon$.
	Now we assume that $$\frac{\log \mu(I_{n}(x))-\log \mu(I_{N_{2i}}(x))}{\log |I_{n}(x)|-\log |I_{N_{2i}}(x)|}=:\alpha_0.$$
	Then we have:
	\begin{equation}\label{under dimloc mu}
	\frac{\alpha_2\log(|I_{N_{2i}}(x)|)+\alpha_0(\log |I_{n}(x)|-\log |I_{N_{2i}}(x)|)}{\log(|I_{N_{2i}}(x)|)+(\log |I_{n}(x)|-\log |I_{N_{2i}}(x)|)}\leq \alpha+\epsilon.
	\end{equation}
	
	Let us illustrate these relationships by figure~\ref{fig:eab--1}
		\begin{figure}[h]  
		\begin{minipage}[t]{0.5\linewidth}  
			\centering  
			\includegraphics[width=0.99\linewidth]{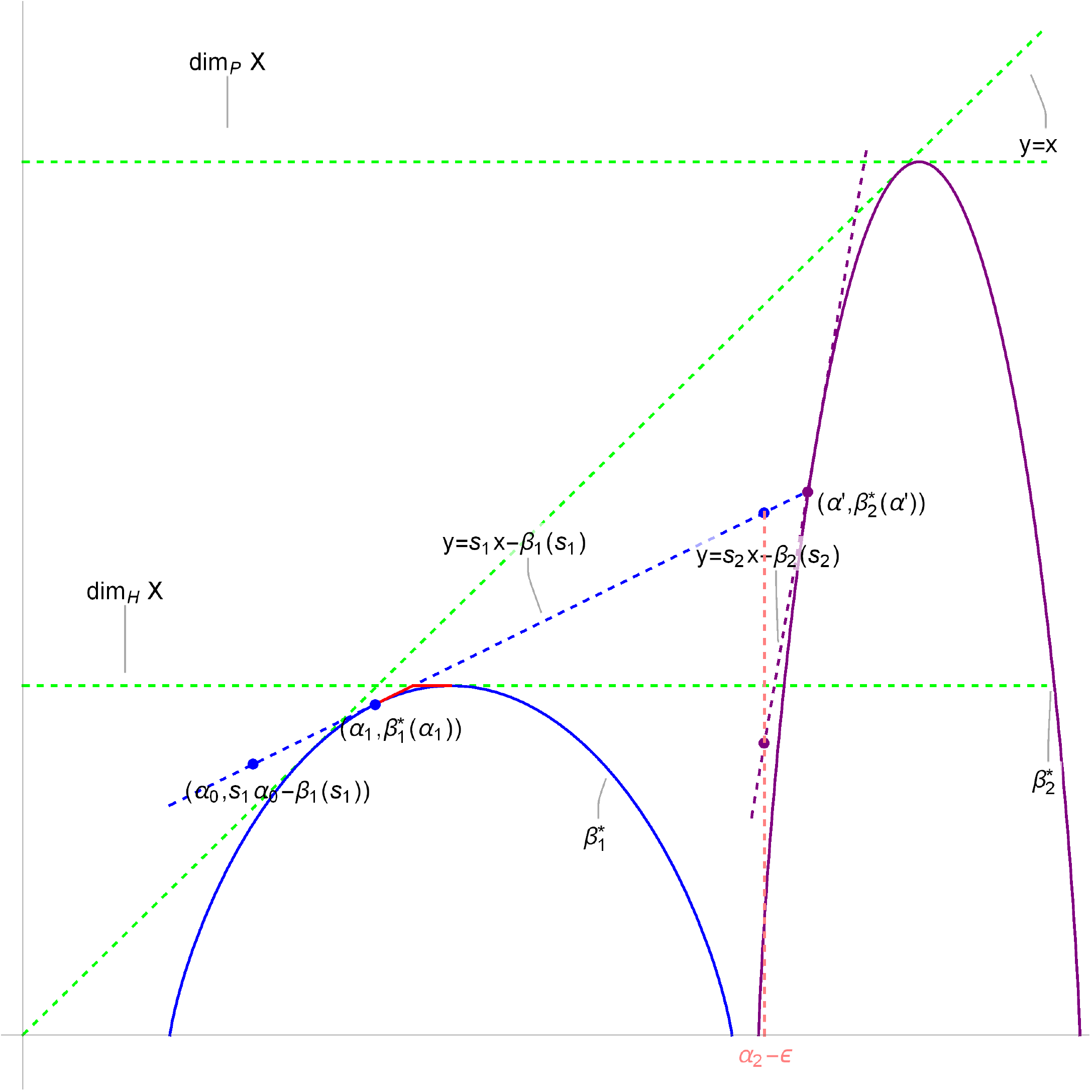}  
			\caption{}  
			\label{fig:eab--1}  
		\end{minipage}
		\begin{minipage}[t]{0.5\linewidth}  
			\centering  
			\includegraphics[width=0.99\linewidth]{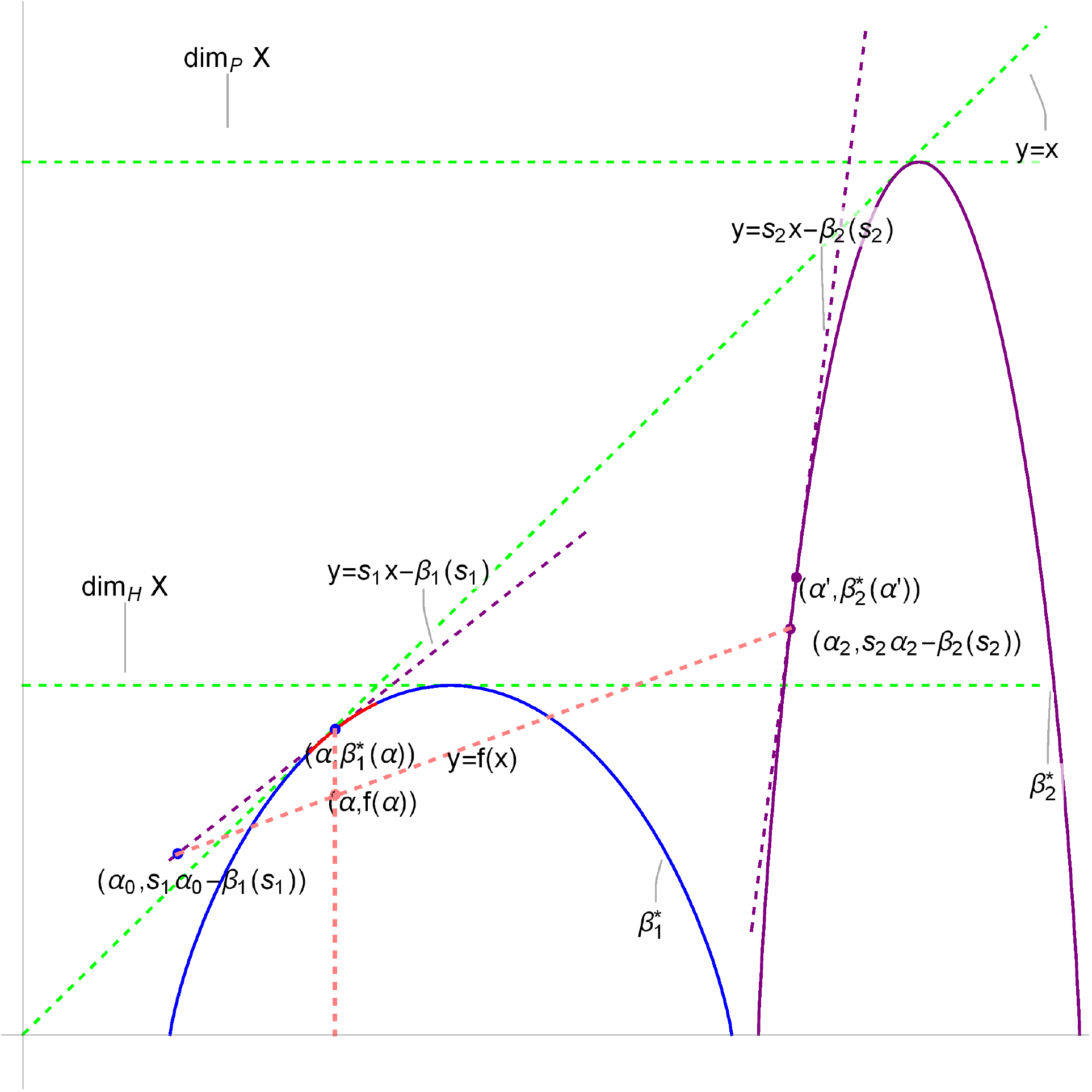}  
			\caption{}  
			\label{fig:eab--2}  
		\end{minipage}  
	\end{figure}
	
	Now we turn to estimate $\log \mu'(I_{N_{2i}}(x))$. First we notice inequality~\eqref{control mu'} and the definition of $\mu'$, one has 
	\begin{eqnarray*}
		\log \mu'(I_{N_{2i}}(x))&\geq&(\log \mu'(I_{N_{2i}}(x))-\log \mu'(I_{N_{2i-1}}(x)))(1+\epsilon)\\
		&=&(s_2(\log \mu(I_{N_{2i}}(x))-\log \mu(I_{N_{2i-1}}(x))))(1+\epsilon)\\
		&&-\beta_2(s_2)(\log |I_{N_{2i}}(x)|-\log |I_{N_{2i-1}}(x)|)(1+\epsilon).
	\end{eqnarray*}
    Second, since $s_2>1$, $\beta_{2}(s_2)>0$ and inequality~\eqref{control mu} and~\eqref{control length},  one has 
	\begin{eqnarray*}
		\log \mu'(I_{N_{2i}}(x))&\geq&s_2(1+\epsilon)\log \mu(I_{N_{2i}}(x))-(1+\epsilon)\beta_2(s_2)\log |I_{N_{2i}}(x)|\\
		&=&s_2\alpha_2(1+\epsilon)\log |I_{N_{2i}}(x)|-(1+\epsilon)\beta_2(s_2)\log |I_{N_{2i}}(x)|\\
		&=&(s_2\alpha_2-\beta_2(s_2))(1+\epsilon)\log |I_{N_{2i}}(x)|
	\end{eqnarray*}
   Third, we notice $\alpha_2\leq \alpha'+\epsilon$, i.e. $\alpha_2-\epsilon\leq \alpha'$, then from figure~\ref{fig:eab--3}, we have   
   $$s_2(\alpha_2-\epsilon)-\beta_2(s_2)\leq s_1(\alpha_2-\epsilon))-\beta_1(s_1)\leq s_1\alpha_2-\beta_1(s_1).$$
   That is $s_2\alpha_2-\beta_2(s_2)\leq s_1\alpha_2-\beta_1(s_1)+S_2\epsilon$.
   So     
   \begin{equation}\label{control mu' N}
  \log \mu'(I_{N_{2i}}(x))\geq(s_1\alpha_2-\beta_1(s_1)+s_2\epsilon)(1+\epsilon)\log |I_{N_{2i}}(x)|.
   \end{equation}
Also from the construction of $\mu'$ one has
	\begin{equation}\label{control mu' n}
	\log \mu'(I_{n}(x))-\log \mu'(I_{N_{2i}}(x))=(s_1\alpha_0-\beta_1(s_1))(\log |I_{n}(x)|-\log |I_{N_{2i}}(x)|)
	\end{equation}
	Then
	\begin{eqnarray*}
		&&d(\mu',x,n)\\
		&=&\frac{\log \mu'(I_n(x))}{\log (|I_n(x)|)}\\
		&=&\frac{\log \mu'(I_{N_{2i}}(x))+\log \mu'(I_{n}(x))-\log \mu'(I_{N_{2i}}(x))}{\log(|I_{N_{2i}}(x)|)+(\log |I_{n}(x)|-\log |I_{N_{2i}}(x)|)}\\
		&\leq&\frac{(s_1\alpha_2-\beta_1(s_1)+s_2\epsilon)(1+\epsilon)\log |I_{N_{2i}}(x)|+(s_1\alpha_0-\beta_1(s_1))(\log |I_{n}(x)|-\log |I_{N_{2i}}(x)|)}{\log(|I_{N_{2i}}(x)|)+(\log |I_{n}(x)|-\log |I_{N_{2i}}(x)|)}\\
		&&(\text{here we use inequality~\eqref{control mu' N} and equation~\ref{control mu' n}})\\
		&\leq&\frac{(s_1\alpha_2-\beta_1(s_1))\log |I_{N_{2i}}(x)|+(s_1\alpha_0-\beta_1(s_1))(\log |I_{n}(x)|-\log |I_{N_{2i}}(x)|)}{\log(|I_{N_{2i}}(x)|)+(\log |I_{n}(x)|-\log |I_{N_{2i}}(x)|)}\\
		&&+s_2\epsilon(1+\epsilon)+(s_1\alpha_2-\beta_1(s_1))\epsilon\\
		&\leq&\frac{s_1\alpha_2\log |I_{N_{2i}}(x)|+s_1\alpha_0(\log |I_{n}(x)|-\log |I_{N_{2i}}(x)|)}{\log(|I_{N_{2i}}(x)|)+(\log |I_{n}(x)|-\log |I_{N_{2i}}(x)|)}-\beta_1(s_1)\\
		&&+s_2\epsilon(1+\epsilon)+(s_1\alpha_2-\beta_1(s_1))\epsilon\\
		&\leq&s_1(\alpha+\epsilon)-\beta_1(s_1)+s_2\epsilon(1+\epsilon)+(s_1\alpha_2-\beta_1(s_1))\epsilon\\
		 && \text{(see inequality~\eqref{under dimloc mu})}
	\end{eqnarray*}
    That is
    \begin{equation}\label{loc dim mu'1}
    d(\mu',x,n)\leq s_1(\alpha+\epsilon)-\beta_1(s_1)+s_2\epsilon(1+\epsilon)+(s_1\alpha_2-\beta_1(s_1))\epsilon
    \end{equation}
    for any $n>N_{2i_0}$ with $d(\mu,x,n)\leq \alpha+\epsilon$.
	
	 Noticing for any $N'\in\mathbb{N}$, there exists $n\geq N'$ with $d(\mu,x,n)\leq \alpha+\epsilon$, so that 
	$$\liminf_{n\to\infty}d(\mu',x,n)\leq s_1\alpha-\beta_1(s_1),$$
	and then $\dl(\mu',x)\leq s_1\alpha-\beta_1(s_1)$ from lemma~\ref{dimloc an}.

	\item  $\alpha'\in[\beta_{2}'(+\infty),\beta_{2}'(1))$ and $\alpha\in[\beta_1'(1),\min\{g(\alpha'),\frac{\log 2}{\log A}\})$.
	The lower bound can be obtained by the following auxiliary measure $\mu'$ defined thought $p_n'$ with
	$$p_n'=\begin{cases}
	A^{\beta_1(s_1)}p^{s_1}&\, N_{2i}<n\leq N_{2i+1}\text{ for some } i\in\mathbb{N},\\
	B^{\beta_2(s_2)}q^{s_2} &\, N_{2i+1}<n\leq N_{2i+2} \text{ for some } i\in\mathbb{N},
	\end{cases}$$

	  For the lower bound, Use the same method we can also define the point $(\alpha_0,s_1\alpha_0-\beta_{1}(s_1))$ and $(\alpha_2,s_2\alpha_2-\beta_{2}(s_2))$.
	  We now replace the tangent in the previous item by the line passing through these two points. We assume the line is the graph of the linear function $y=f(x)$.  
	Let us illustrate these relationships by figure~\ref{fig:eab--2}.
	
	The proof is almost the same as before to get that for any $x\in E(\mu,\alpha,\alpha')$ we have $\dl(\mu',x)\leq f(\alpha)$, but we also need to use the inequality 
	$$f(\alpha)\leq \beta_1^*(\alpha).$$
	Also $\beta_2^*(\alpha')$ is a  travail bound since $E(\mu,\alpha,\alpha')\subset \Eu(\mu,\alpha')$.
	So that the upper bound can be obtained. 	
	
\end{itemize}
\item Now we turn to $\dim_P E(\mu,\alpha,\alpha')$. We assume that $\beta_1'(s_1)=\alpha, \beta_2'(s_2)=\alpha'$
\begin{itemize}
	\item $(\alpha,\alpha')\notin (\beta'_1(1),\beta'_1(-\infty)]\times[\beta'_2(+\infty),\beta'_2(1))$:\\
	If $\alpha'\notin [\beta'_2(+\infty),\beta'_2(1))$, let $\beta_2^*(\alpha_0)=\max\{\beta_2^*(\alpha_1):\alpha_1\leq \alpha'\},$ we can assume that $\alpha_0=\beta_2'(s_0)$. Using the same way as the proof of theorem~\ref{main},
	define $$p_n'=\begin{cases}
	q^{s_1} A^{\beta_{1}(s_1)}&\, N_{2i-1}<n\leq N_{2i} \text{ for some } i\in\mathbb{N},\\
	B^{\beta_2(s_0)}p^{s_0}&\, N_{2i}<n\leq N_{2i}'\text{ for some } i\in\mathbb{N},\\
	B^{\beta_2(s)}p^{s}&\, N_{2i}'<n\leq N_{2i+1}\text{ for some } i\in\mathbb{N}.
	\end{cases}$$
	Where $\{N_{2i}'\}_{i\in\mathbb{N}}$ is a sequence of numbers such that
	$N_{2i}<N_{2i}'\leq N_{2i+1}$, $\lim_{i\to\infty}\frac{N_{2i}'}{N_{2i}}=+\infty$ and 
	$\lim_{i\to\infty}\frac{N_{2i+1}}{N_{2i}'}=+\infty$.
	This can be done since $\lim_{i\to \infty}\frac{N_{i+1}}{N_{i}}=+\infty$. We now define a measure $\mu'$ and can prove that 
	$\mu'(E(\mu,\alpha,\alpha'))=1$ and  for $\mu'$-a.e. $x\in X$, $\du(\mu',x)\geq \beta_2^*(\alpha_0)$.
	Then, from \cite[Proposition 2.3]{Falconer1997}, we have $\dim_P E(\mu,\alpha,\alpha')\geq \beta_2^*(\alpha_0).$
	Since $E(\mu,\alpha,\alpha')\subset\Eu(\alpha')$ and  $\dim_P \Eu(\mu,\alpha')= \beta_2^*(\alpha_0),$ we get that $$\dim_P E(\mu,\alpha,\alpha')= \dim_P \Eu(\mu,\alpha').$$
	
	If $\alpha'\in [\beta'_2(+\infty),\beta'_2(1))$ and $\alpha\notin(\beta'_1(1),\beta'_1(-\infty)]$,
	define $$p_n'=\begin{cases}
	p^{s_1}A^{\beta_1(s_1)}&\, N_{2i}<n\leq N_{2i}' \text{ for some } i\in\mathbb{N}\\
	p&\, N_{2i}'<n\leq N_{2i+1} \text{ for some } i\in\mathbb{N},\\
	q&\, N_{2i+1}<n\leq N_{2i+1}'\text{ for some } i\in\mathbb{N},\\
	q^{s_2} B^{\beta_2(s_2)}&\, N_{2i+1}'<n\leq N_{2i+2}\text{ for some } i\in\mathbb{N}.
	\end{cases}$$
	Where $\{N_{i}'\}_{i\in\mathbb{N}}$ is a sequence of numbers such that
	$N_{i}<N_{i}'\leq N_{i+1}$, $\lim_{i\to\infty}\frac{N_{i}'}{N_{i}}=+\infty$ and 
	$\lim_{i\to\infty}\frac{N_{i+1}}{N_{i}'}=+\infty$.
	This can be done since $\lim_{i\to \infty}\frac{N_{i+1}}{N_{i}}=+\infty$.
	
	We now define a measure $\mu'$, a similar method for the lower bound for $\dim_P \Eu(\mu,\alpha')$ and $E(\mu,\alpha,\alpha')\subset \Eu(\mu,\alpha')$ yield 
	$\dim_{P}E(\mu,\alpha,\alpha')=\dim_{P} \Eu(\mu,\alpha')$.
	\item For $\alpha\in(\beta'_1(1),\beta'_1(-\infty)]$ and $\alpha'\in[\beta_2'(+\infty),h(\alpha))$. Recall that we take the tangent to the graph $\beta_2^*$ passing through the point $(\alpha,\beta_1^*(\alpha))$ and denote the point of tangency by $(h(\alpha),\beta_2^*(h(\alpha)))$.
	Then take $s_2\in(1,+\infty]$ such that $\beta_2'(s_2)=h(\alpha)=:\alpha_2$ and 
	$s_1\in [-\infty,1)$ such that $\beta_1'(s_1)=\alpha$.
	$\dim_P E(\mu,\alpha,\alpha')\geq s_2\alpha'-\beta_2(s_2)$ can be obtained by using the following auxiliary measure $\mu'$ though $\{p_n'\}_{n\in\mathbb{N}}$ as the lower bound for $\dim_P \Eu(\mu,\alpha')$.
	$$p_n'=\begin{cases}
	p^{s_1}A^{\beta_1(s_1)}&\, N_{2i}<n\leq N_{2i+1} \text{ for some } i\in\mathbb{N}\\
	q&\, N_{2i+1}<n\leq N_{2i+1}'\text{ for some } i\in\mathbb{N},\\
	q^{s_2} B^{\beta_2(s_2)}&\, N_{2i+1}'<n\leq N_{2i+2}\text{ for some } i\in\mathbb{N}.
	\end{cases}$$
	
	The upper bound is similar with the Hausdorff dimension in the second item, but noticing figure~\ref{fig:eab--3}. We need to show that for any $x\in E(\mu,\alpha,\alpha')$ we have $\du(\mu',x)\leq s_2\alpha'-\beta_2(s_2)$. 
			\begin{figure}[h]  
		\begin{minipage}[t]{0.50\linewidth}  
			\centering  
			\includegraphics[width=0.99\linewidth]{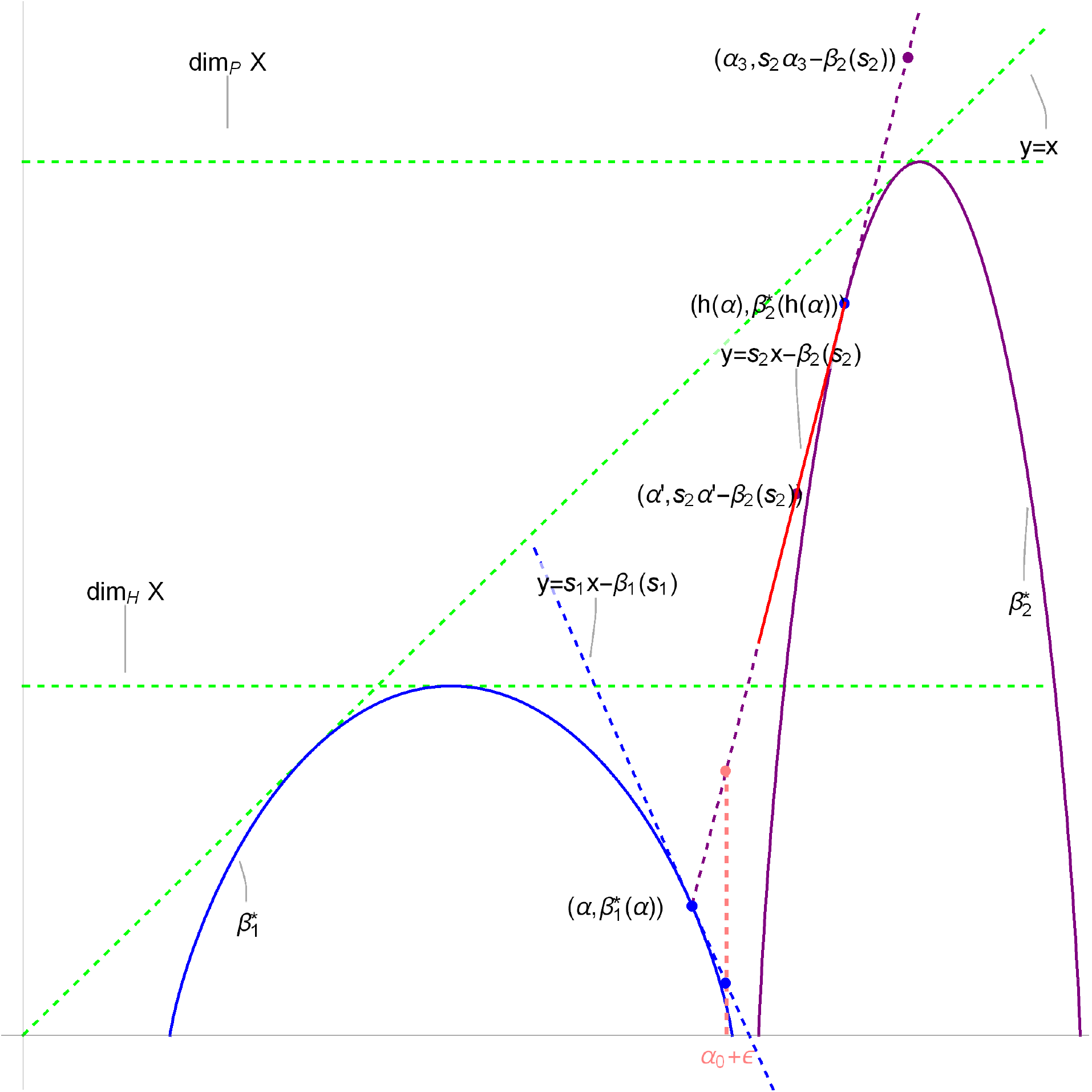}  
			\caption{}  
			\label{fig:eab--3}  
		\end{minipage}
		\begin{minipage}[t]{0.50\linewidth}  
			\centering  
			\includegraphics[width=0.99\linewidth]{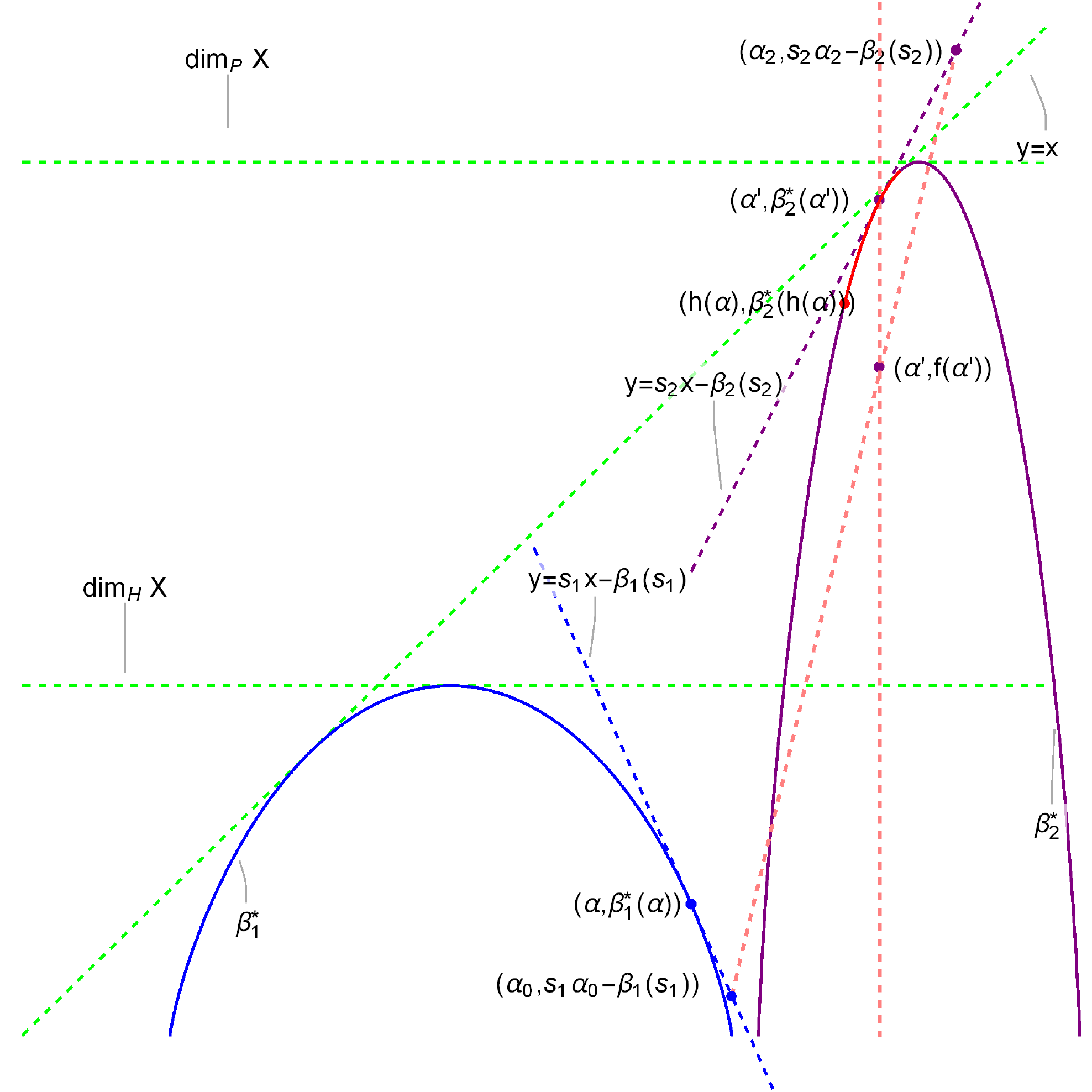}  
			\caption{}  
			\label{fig:eab--4}  
		\end{minipage}  
	\end{figure}	

    Fix $x\in E(\mu,\alpha,\alpha')$, that is for any $\epsilon>0$, there exists $N\in\mathbb{N}$ such that for any $n\geq N$,
    $d(\mu,x,n)\geq \alpha-\epsilon$ and $d(\mu,x,n)\leq \alpha'+\epsilon$.

    Choose $i_0$ with $N_{2i_0+1}>N$ large enough such that for any $i\geq i_0$ one has
    \begin{equation}\label{control mu-packing}
    1\leq\frac{\log \mu(I_{N_{2i+1}}(x))}{\log \mu(I_{N_{2i+1}}(x))-\log \mu(I_{N_{2i}}(x))}\leq 1+\epsilon
    \end{equation}
    \begin{equation}\label{control mu'-packing}
    1\leq\frac{\log \mu'(I_{N_{2i+1}}(x))}{\log \mu'(I_{N_{2i+1}}(x))-\log \mu'(I_{N_{2i}}(x))}\leq 1+\epsilon
    \end{equation}
    \begin{equation}\label{control length-packing}
    1\leq\frac{\log |I_{N_{2i+1}}(x)|}{\log |I_{N_{2i+1}}(x)|-\log |I_{N_{2i}}(x)|}\leq 1+\epsilon
    \end{equation}
    
    For any $n>N_{2i_0+1}$, one has $d(\mu,x,n)\leq \alpha'+\epsilon$, choose the largest $i$ with $N_{2i+1}\leq n$, one has $\alpha_0:=d(\mu,x,N_{2i+1})\geq \alpha-\epsilon$.
    Now we assume that $$\frac{\log \mu(I_{n}(x))-\log \mu(I_{N_{2i+1}}(x))}{\log |I_{n}(x)|-\log |I_{N_{2i+1}}(x)|}=:\alpha_3.$$
    Then we have:
    \begin{equation}\label{under dimloc mu-packing}
    \frac{\alpha_3\log(|I_{N_{2i+1}}(x)|)+\alpha_0(\log |I_{n}(x)|-\log |I_{N_{2i+1}}(x)|)}{\log(|I_{N_{2i+1}}(x)|)+(\log |I_{n}(x)|-\log |I_{N_{2i+1}}(x)|)}\leq \alpha'+\epsilon.
    \end{equation}
    
    Now we turn to estimate $\log \mu'(I_{N_{2i+1}}(x))$. 
    
    First, we notice inequality~\eqref{control mu'-packing} and the definition of $\mu'$, one has 
    \begin{eqnarray*}
    	\log \mu'(I_{N_{2i+1}}(x))&\geq&(\log \mu'(I_{N_{2i+1}}(x))-\log \mu'(I_{N_{2i}}(x)))(1+\epsilon)\\
    	&=&(s_1(\log \mu(I_{N_{2i+1}}(x))-\log \mu(I_{N_{2i}}(x))))(1+\epsilon)\\
    	&&-\beta_1(s_1)(\log |I_{N_{2i+1}}(x)|-\log |I_{N_{2i}}(x)|)(1+\epsilon).
    \end{eqnarray*}
    Second, from inequality~\eqref{control mu-packing} and~\eqref{control length-packing},  one has 
    \begin{eqnarray*}
    	\log \mu'(I_{N_{2i+1}}(x))&\geq&(s_1+|s_1|\epsilon)\log \mu(I_{N_{2i+1}}(x))-(1+\epsilon)\beta_1(s_1)\log |I_{N_{2i+1}}(x)|\\
    	&=&(s_1+|s_1|\epsilon)\alpha_0\log |I_{N_{2i}}(x)|-(1+\epsilon)\beta_1(s_1)\log |I_{N_{2i+1}}(x)|\\
    	&=&((s_1\alpha_0-\beta_1(s_1))+|s_1|\epsilon\alpha_0-\beta_1(s_1)\epsilon)\log |I_{N_{2i}}(x)|
    \end{eqnarray*}
    Third, we notice $\alpha_0\geq \alpha-\epsilon$, i.e. $\alpha_0+\epsilon\geq  \alpha$,  then from figure~\ref{fig:eab--3}, we have 
    $$s_1(\alpha_0+\epsilon)-\beta_1(s_1)\leq s_2(\alpha_0+\epsilon))-\beta_2(s_2)\leq s_2\alpha_0-\beta_2(s_2)+s_2\epsilon.$$
    That is $s_1\alpha_0-\beta_1(s_1)\leq s_2\alpha_0-\beta_2(s_2)+s_2\epsilon+|s_1|\epsilon$.
    So     
    \begin{equation}\label{control mu' N-packing}
    \log \mu'(I_{N_{2i+1}}(x))\geq(s_2\alpha_0-\beta_2(s_2)+(s_2+|s_1|(1+\alpha_0)+\beta_1(s_1))\epsilon)\log |I_{N_{2i+1}}(x)|.
    \end{equation}
    Also from the construction of $\mu'$ and the definition of $\alpha_3$, one has
    \begin{equation}\label{control mu' n-packing}
    \log \mu'(I_{n}(x))-\log \mu'(I_{N_{2i+1}}(x))=(s_2\alpha_3-\beta_2(s_2))(\log |I_{n}(x)|-\log |I_{N_{2i+1}}(x)|)
    \end{equation}
    Then, using a similar method in the proof of inequality \eqref{loc dim mu'1}, but replace ~\eqref{control mu' N},~\eqref{control mu' n} and ~\eqref{under dimloc mu} by ~\eqref{control mu' N-packing},~\ref{control mu' n-packing}and~\eqref{under dimloc mu-packing},we can conclude
    \begin{equation}\label{loc dim mu'2}
    d(\mu',x,n)\leq s_2(\alpha'+\epsilon)-\beta_2(s_2)+s_2\epsilon(1+\epsilon)+(s_1\alpha_3-\beta_1(s_1))\epsilon
    \end{equation}
    for any $n\geq N_{2i_0+1}$. 
    This gives 
    $$\limsup_{n\to\infty}d(\mu',x,n)\leq s_2\alpha'-\beta_2(s_2),$$
    and then $\du(\mu',x)\leq s_2\alpha'-\beta_2(s_2)$ from lemma~\ref{dimloc an}.

	\item For $\alpha\in(\beta'_1(1),\beta'_1(-\infty)]$ and $\alpha'\in[h(\alpha),\beta_2'(1))$.  For the lower bound, use the same method as the previous item, we can also define the point $(\alpha_0,s_1\alpha_0-\beta_{1}(s_1))$ and $(\alpha_2,s_2\alpha_2-\beta_{2}(s_2))$.
	We now replace the tangent in the previous item by the line passing through these two points. We assume the line is the graph of the linear function $y=f(x)$.   
	Let us illustrate these relationships by figure~\ref{fig:eab--4}.
	
	The proof is almost the same as before to get that for any $x\in E(\mu,\alpha,\alpha')$ we have $\du(\mu',x)\leq f(\alpha')$, but we also need to use the inequality 
	$$f(\alpha')\leq \beta_2^*(\alpha').$$
	
 The upper bound can be obtained since $f(\alpha')\leq \beta_2^*(\alpha')$.

For the lower bound, we goes to the auxiliary measure $\mu'$ though $\{p_n'\}_{n\in\mathbb{N}}$ as the lower bound for $\dim_P \Eu(\mu,\alpha')$.
$$p_n'=\begin{cases}
p^{s_1}A^{\beta_1(s_1)}&\, N_{2i}<n\leq N_{2i+1} \text{ for some } i\in\mathbb{N}\\
q^{s_2} B^{\beta_2(s_2)}&\, N_{2i+1}<n\leq N_{2i+2}\text{ for some } i\in\mathbb{N}.
\end{cases}$$	
\end{itemize}
\end{enumerate}
  	
\end{proof}
Let us give some remarks to the proof of theorem~\ref{main2}.
\begin{enumerate}
	\item The Hausdorff dimension: The lower bound can be obtained by a suitable auxiliary measure $\mu'$. Now we turn to the upper bound.
	When $(\alpha,\alpha')\notin [\beta'_1(1),\beta'_1(0)]\times[\beta'_2(+\infty),\beta'_2(1))$, the upper bound can be obtained by $E(\mu,\alpha,\alpha')=\El(\mu,\alpha)\cap\Eu(\mu,\alpha')$ and theorem~\ref{main}.
	But this may not give the sharp upper bound for $\alpha'\in[\beta'_2(+\infty),\beta'_2(1))$, also the methods in \cite{Olsen1995,Barral2015} fail. So we turn to use \cite[Proposition 2.3]{Falconer1997} and calculate the upper bound of $\dl(\mu',x)$ for all $x\in E(\mu,\alpha,\alpha')$ for a suitable measure $\mu'$ to give the sharp upper bound.
	
	This result is reasonable, if we compare with $\dim_{H}(\El(\mu,\alpha))=\min\{\alpha,\frac{\log 2}{\log A}\}$  when $\alpha\in [\beta_{1}'(1),\beta_{1}'(0))$ since we use the optimal tangent $y=x$. But for the situation $E(\mu,\alpha,\alpha')$, it is $y=s_1x-\beta_1(s_1)$. By the way, when $\alpha'$ fulfill $[\beta'_2(+\infty),\beta'_2(1))$, the curve $\dim_H E(\mu,\alpha,\alpha')$ fulfill the gap between $\dim_{H}(\El(\mu,\alpha))$ and $\beta_1^*(\alpha)$ for $\alpha\in [\beta_{1}'(1),\beta_{1}'(0))$. 
	
	\item The packing dimension: It also has a duality with the Hausdorff dimension so there is a similar result. The lower bound can be obtained by a suitable auxiliary measure $\mu'$. Now we turn to the upper bound.
	When $(\alpha,\alpha')\notin [\beta'_1(1),\beta'_1(-\infty)]\times[\beta'_2(+\infty),\beta'_2(1))$, the upper bound can be obtained by $E(\mu,\alpha,\alpha')=\El(\mu,\alpha)\cap\Eu(\mu,\alpha')$ and theorem~\ref{main}.
	But this may not give the sharp upper bound for $\alpha\in[\beta'_1(1),\beta'_1(-\infty)]$, also the methods in \cite{Olsen1995,Barral2015} fail. So we turn to use \cite[Proposition 2.3]{Falconer1997} and calculate the upper bound of $\du(\mu',x)$ for all $x\in E(\mu,\alpha,\alpha')$ for a suitable measure $\mu'$ to give the sharp upper bound.
	
	This result is also reasonable, if we compare with $\dim_{P}(\El(\mu,\alpha'))=\alpha'$  when $\alpha'\in [\beta_2'(+\infty),\beta'_2(1))$ since we use the optimal tangent $y=x$. But for the situation $E(\mu,\alpha,\alpha')$, it is $y=s_2x-\beta_2(s_2)$. By the way, when $\alpha$ fulfill $[\beta'_1(1),\beta'_1(-\infty)]$, the curve $\dim_H E(\mu,\alpha,\alpha')$ fulfill the gap between $\dim_{H}(\El(\mu,\alpha))$ and $\beta_2^*(\alpha)$ for $\alpha'\in [\beta_2'(+\infty),\beta'_2(1))$.      
\end{enumerate} 


\end{document}